\documentclass[reqno,a4paper]{amsart}
\usepackage{bm}
\usepackage{hyperref}
\usepackage{amssymb}
\usepackage{amsmath,amsthm, latexsym, wasysym}
\usepackage{enumerate}
\usepackage{amsfonts}
\newtheorem{theorem}{Theorem}[section]
\newtheorem{lemma}[theorem]{Lemma}

\newtheorem{proposition}[theorem]{Proposition}

\theoremstyle{definition}
\newtheorem{definition}[theorem]{Definition}

\numberwithin{equation}{section}
\usepackage[T1]{fontenc}
\usepackage{mathrsfs}
\usepackage{setspace}
\usepackage{graphicx,subfigure}
\usepackage{xcolor}
\usepackage{epstopdf}
\usepackage[utf8]{inputenc}
\usepackage{multicol}
\usepackage{amsmath}
\usepackage{amsthm}
\usepackage{upgreek}
\usepackage{amsfonts}
\usepackage{contour}
\usepackage{float}
\usepackage{breqn}
\usepackage{textcomp}
\usepackage{gensymb}
\usepackage[left=.9in,right=1in,top=.8in,bottom=.8in]{geometry}
\usepackage{hyperref}
\usepackage{lipsum}
\usepackage{apptools}
\newcommand{\remove}[1]{}
\theoremstyle{remark}

\newcommand{\ra} {\rightarrow}
\newcommand{\RR} {\mathbb{R}}
\newcommand{\N}{{\mathbb N}}
\newcommand{\DD} {\displaystyle}
\newcommand{\la} {\lambda}
\newcommand{\Om}{\Omega}
\theoremstyle{remark}

\newcommand{\al}{\alpha}
\newcommand{\ba}{\beta}

\newcommand{\noi}{\noindent}
\newcommand {\n}{\nonumber\\}

\newcommand{\ol}{\overline}
\newcommand{\e}{\epsilon}
\newcommand{\p}{p_s^*}
\newcommand{\fpl}{(-\Delta)_p^s\,}

\begin{document}


\title{Regularity results for Choquard equations involving fractional $p$-Laplacian}
\author[Reshmi Biswas and Sweta Tiwari]
{Reshmi Biswas and Sweta Tiwari}
\address{Reshmi Biswas \newline
	Department of Mathematics, IIT Guwahati, Assam 781039, India}
\email{b.reshmi@iitg.ac.in}

\address{Sweta Tiwari \newline
	Department of Mathematics, IIT Guwahati, Assam 781039, India}
\email{swetatiwari@iitg.ac.in}


\begin{abstract}{	\noindent	In this article, first we address the regularity of weak solution for a class of $p$-fractional Choquard equations:	
	\begin{equation*}
	\;\;\;	\left.\begin{array}{rl}	
	(-\Delta)_p^su&=\left(\displaystyle\int_\Omega\frac{F(y,u)}{|x-y|^{\mu}}dy\right)f(x,u),\hspace{5mm}x\in \Omega,\\\\
	u&=0,\hspace{35mm}x\in \mathbb R^N\setminus \Omega,
	\end{array}
	\right\}
	\end{equation*}
	where $\Omega\subset\mathbb R^N$ is a smooth  bounded   domain, $1<p<\infty$ and $0<s<1$ such that $sp<N,$ $0<\mu<\min\{N,2sp\}$ and $f:\Omega\times\mathbb R\to\mathbb R$ 
	is a continuous function with at most critical growth condition (in the sense of Hardy-Littlewood-Sobolev inequality) and $F$ is its primitive.
	Next, for  $p\geq2,$ we discuss the
	Sobolev versus H\"{o}lder  minimizers of the energy functional $J$ associated to the above problem, and using that we establish the 
	existence of the local minimizer of $J$ in the fractional Sobolev space $W_0^{s,p}(\Omega).$
	Moreover, we discuss the aforementioned results by adding 
		a local perturbation term (at most critical in the sense of Sobolev inequality) 
		in the right-hand side in the  above equation.}
\end{abstract}


\maketitle


\section{Introduction}
Our first aim in this article is to study the following  doubly nonlocal $p$-fractional  Choquard equation :
\begin{equation}\label{mainprob}
\;\;\;	\left.\begin{array}{rl}	
(-\Delta)_p^su&=\left(\DD\int_\Om\frac{F(y,u)}{|x-y|^{\mu}}dy\right)f(x,u),\hspace{5mm}x\in \Om,\\\\
u&=0,\hspace{40mm}x\in \RR^N\setminus \Om,
\end{array}
\right\}
\end{equation}
where $\Om\subset\RR^N$ is a  bounded   domain with $C^{1,1}$ boundary, $1<p<\infty$ and $0<s<1$ such that $sp<N,\,$ $0<\mu<\min\{N,2sp\}$ and $f:\Om\times\RR\to\RR$ 
is a continuous function with at most critical growth condition (in the sense of Hardy-Littlewood-Sobolev inequality), described later.
Here $F(x,t)=\int_{0}^{t}f(x,\tau)d\tau$ is the primitive of $f$.	
The nonlocal operator  $(-\Delta)_{p}^{s}$ is defined  as
\begin{equation}\label{operator}
(-\Delta)_{p}^{s} u(x):=  2 \DD\lim_{\e\to 0^+}
\int_{\RR^N\setminus {B_\e(x)}}\frac{\mid
	u(x)-u(y)\mid^{p-2}(u(x)-u(y))}{\mid
	x-y\mid^{N+sp}}dy, ~~x \in \RR^N
\end{equation} up to a normalized constant.\\\\
In recent years,   significant attention has been given in the study of the problems   involving the nonlocal operators, 
due to its various applications in the real world such as thin obstacle
problems, finance, conservation laws, phase transition, crystal dislocation, anomalous diffusion, material
science, etc. (see for e.g., \cite{ba,caff2,lev,silves} and the references therein for more details).
One can refer  to the  monograph \cite{bisci} for the study of nonlocal problems driven by  the fractional Laplacian and 
\cite{hit,sarika} for the detailed discussions on  the fractional $p$-Laplacian and problems involving it.\\ 
\par On the other hand, the study of Choquard type equations was started 
with the celebrated work of Pekar \cite{pekar}, where the author considered 
the following nonlinear Schr\"{o}dinger-Newton equation:
{\begin{align}\label{sn}
	-\Delta u + V(x)u = ({\mathcal{K}}_\mu * u^2)u +\la f(x, u),
	\end{align}}
\noi	where $\mathcal{K}_\mu$ denotes  the  Riesz potential.
The nonlinearity in the right-hand side of \eqref{sn} is termed as Hartree-type nonlinearty. This type of nonlinearity plays a key role in the study the Bose-Einstein
condensation (see \cite{bose}) and also describes the self gravitational
collapse of a quantum mechanical wave function (see \cite{penrose}). 
For $V(x)=1,\la=0$, the equations of type \eqref{sn} were extensively  studied in  \cite {lieb,lions}. 
For more results on the existence of solutions of Choquard equations, without attempting to provide a complete list, 
we refer to \cite{moroz,moroz-main,moroz4,moroz5} and the references therein. In the fractional Laplacian set up, Wu \cite{wu} discussed existence
and stability of solutions for the equations
{\begin{align}\label{sn1}
	(-\Delta)^s u + \omega u = ({\mathcal{K}}_\mu *|u|^q)|u|^{q-2}u + \la f(x,u) ~~\text{in} ~\RR^N ,\end{align}}
where $q = 2, ~\la = 0$ and $\mu\in (N -2s, N).$  
For the
critical case, i.e., $q = 2_{\mu,s}^*:=(2N -{2{\mu}/2})/(N - 2s),$ Mukherjee and Sreenadh \cite{tuhina} studied existence and multiplicity, and regularity results for the
solutions of \eqref{sn1}  in a smooth bounded domain for $w=0$ and $f(x, u) = u.$  Pucci et al. \cite{pucci}
studied some Schr\"odinger-Choquard-Kirchhoff equation driven by the fractional $p$-Laplacian with critical Hardy-Littlewood-Sobolev exponent.
For more details regarding Choquard type equations, we refer to the survey  paper \cite{ts} and references therein.\\

\par The regularity of weak solutions has been one of the most interesting topics since years and the
literature available on the regularity of weak solutions for both local and non-local problems is quite vast. For the regularity results of the local elliptic problems,
we refer to \cite{dib, liberman, tol}.
A systematic study on the regularity results of the non-local elliptic problems involving fractional Laplacian started with the pioneering work of Caffarelli and Silvestre in \cite{caffe1}.
Consider the following non-local  problem: 
\begin{equation}\label{nl}
(-\Delta)_p^s u =g\text{ in } \Om, \quad
u=h \text{ in } \RR^N\setminus \Om. 
\end{equation} When $p=2,$ in \cite{caffe1}, Caffarelli and Silvestre established the interior $C^{1+\al},$ $\al>0,$ regularity for viscosity solutions to  \eqref{nl}.
The authors also proved  interior $C^{2s+\al}$ regularity for the convex equation (see \cite{caffe2}). 
For the regularity of weak solutions to  free boundary problem involving the fractional Laplacian $(p=2),$ we refer to \cite{silves}.  
Concerning the boundary regularity for the solution of \eqref{nl}, for $p=2,\,h=0$ and  $g \in L^\infty(\Om),$
we refer to the work of Ros-Oton and Serra in \cite{RS}. Here the authors used a 
barrier function and the interior regularity results for the fractional Laplacian to
show that any weak solution $u$ of \eqref{nl} belongs to $C^s(\mathbb R^N)$ and
$\frac{u}{d^s}{|_\Om}\in C^{\alpha},$ up to the boundary $\partial\Om,$ for some $\alpha\in(0,1)$.
In \cite{RS1}, the authors discussed  the  high integrability  of these weak solution by using the 
regularity of Riesz potential established in  \cite{stein}. 
The regularity results for the non-local quasi-linear problem is explored by Squassina et al. in \cite{squa}, where the authors
studied the global H\"older regularity for the weak solutions to \eqref{nl},
for $p\in(1,\infty),\, h=0,$ and $g\in L^{\infty}(\Om)$.
Also, regarding the fine boundary regularity results for the problems of type \eqref{nl}, for the degenerate case  $(p\geq2)$ and $h=0$, we cite \cite{fine-bd}. 
Here the authors exhibited a weighted H\"{o}lder regularity up to the boundary, that is,
$\frac{u}{d^s}{|_\Om}\in C^{\alpha,}$ up to the boundary $\partial\Om,$ for some $\alpha\in(0,1)$.
We would like to mention that
the fine boundary regularity for the singular case $(1<p<2)$ is still an open problem.\\
\par		Concerning the regularity of the Choquard equations, we refer to \cite{yangjmaa}, 
in which Gao and Yang studied the Dirichlet problem involving local Laplacian 
and the critical  Choquard type
nonlinearity (in the sense of Hardy-Littlewood-Sobolev inequality).
Moroz and Schaftingen \cite{moroz-main} established the
$W^{2,q}_{\text{loc}}(\mathbb R^N)$-regularity $(q>1)$ of the weak solutions to the following 
Choquard problem involving local Laplacian:
\begin{align}\label{mrz}
-\Delta u +   u = (\mathcal K_\mu * F(u))f(u)\;\; \text{in\;\;} \RR^N ,
\end{align} where
$$|tf(t)|<C(|t|^{\frac{N+\mu}{N}}+|t|^{\frac{N+\mu}{N-2}}), \text{\;\;\; for some constant\;\;} C>0.$$
\par Although an extensive research is done on the existence of solutions for the doubly non-local problems, there are very few results
present in the literature regarding the regularity of weak solutions to such problems.
By generalizing the idea of \cite{moroz4}, in \cite{avenia}, for the fractional Laplacian framework, the authors established
the regularity results  for solutions of the following Choquard equation  :
$$ (-\Delta)^s u +  \omega u = (\mathcal K_\mu * |u|^r)|u|^{r-2}u,\;\; u \in H^s(\RR^N ),$$ where  $\omega > 0,$ 
$ N\geq 3,$ $\mu \in (0, N),$ $s \in(0, 1),$ and $\tilde 2_{s,\mu}^*<r<2_{s,\mu}^*.$
In \cite{su}, the authors studied the $L^\infty(\mathbb R^N)$ bound of the non-negative ground state solution to some Kirchhoff-Choquard equation driven by the
fractional Laplacian  with critical Choquard term (in the sense of Hardy-Littlewood-Sobolev inequality). Very recently, Giacomoni et al. \cite{divya} studied the regularity result
for the following generalized doubly non-local problem  in a smooth bounded domain $\Om$ in $\RR^N$: 
\begin{equation}\label{dg}
\left.\begin{array}{rllll}
(-\Delta)^s u
=g(x,u)+ \left(\DD \int_{\Om}\frac{F(u)(y)}{|x-y|^{\mu}}dy\right) f(u)(x)  \; \text{in}\;
\Om,\quad u=0\; \text{ in } \RR^N \setminus \Om,
\end{array}
\right\}
\end{equation}
where $f:\mathbb R \rightarrow \mathbb R$  is a continuous  function such that there exists a constant $C>0,$
\begin{align*}
|tf(t)| \leq C(|t|^{\frac{2N-\mu}{N}}+ |t|^{\frac{2N-\mu}{N-2s}})
\end{align*} and $ g: \overline{\Om} \times \mathbb{R} \ra \mathbb{R}$ is  a Carath\'eodory function  satisfying Sobolev type critical (or singular) growth assumption.
\par We mention that the techniques used in \cite{divya, moroz-main}
cannot be  implemented straightforward to \eqref{mainprob} due to lack of Hilbert nature of the solution space associated to the problem.
The regularity result  for the quasilinear Choquard equations involving the 
local (or fractional) $p$-Laplacian are very few.  For instance, consider the following equation studied in \cite{miy}: 
\begin{align}\label{miy}
(-\Delta)_p^s u +  \omega u = \left(\frac{1}{|u|^{\mu}} * F(u)\right)f(u)\;\;  \text{in\;\;} \RR^N, 
\end{align} where $\omega>0$ is a real number and $f$ has sub-critical growth in terms of  Hardy-Littlewood-Sobolev inequality.
For the case $s=1,$ we cite \cite{alves}, in which the authors studied \eqref{miy} in the local $p$-Laplacian set up. In both the aforementioned works,  the authors 
proved local H\"older regularity of the weak solutions of \eqref{miy} with some restrictive conditions, viz., $\mu<sp$ and $\mu<p,$ respectively.\\ 
\par\noi Inspired by all these works,   by using a unified boot-strap technique
for $1<p<\infty,$ first we investigate $a~priori$ bound for the weak solutions to \eqref{mainprob} which covers
a large class of nonlinearities (up to the critical level in the sense of Hardy-Littlewood-Sobolev inequality). 
After achieving  $L^\infty(\Om)$ estimate on the weak solution to \eqref{mainprob}, we use the result 
by Squassina et al. \cite{squa} along with Hardy-Littlewood-Sobolev inequality, to infer the H\"{o}lder regularity result. 
To the best of our knowledge, the $L^\infty(\Om)$ bound on 
the weak solutions to the doubly non-local problem of type \eqref{mainprob} involving critical Choquard type nonlinearity is established for the first time in this present work.   
\par Next, we discuss the Sobolev versus H\"older minimizers for the  energy functional associated to \eqref{mainprob}. We show that  local
minimizers of the energy functional associated to \eqref{mainprob}
with respect to $C_d^0(\overline{\Om})$-topology  are also local minimizers of the same energy functional with respect to $W_0^{s,p}(\Om)$-topology.
In variational problems, this result plays  an important  role in establishing 
the  multiplicity of solutions.
In the local framework, Brezis and Nirenberg \cite{niren} were the first
to study this type of result, where the authors showed that the local minima of the associated energy functional in $C^1$ topology
and in $H^1$, topology coincides.   In the fractional framework $(p=2)$, the analogous result is proved in \cite{squassina}.
In \cite{SH}, this result is further generalized for the fractional $p$-Laplacian set up for $p\geq2$. For non-local nonlinearity, 
Gao and Yang \cite{yangjmaa}  studied such result for the following Brezis-Nirenberg type  Critical Choquard problem involving local
Laplacian under some appropriate assumptions on $f$: $$
-\Delta u=\la f(u)+ \left(\DD \int_{\Om}\frac{|u(y)|^{2_\mu^*}}{|x-y|^{\mu}}dy\right) |u(x)|^{2_\mu^*-2}u \text{\;\; in\;\;}\Om,\;\;\;
u=0\; \text{ in\;\; } \partial \Om,$$
where $\Om\subset \RR^N$, $N\geq3$ is a bounded domain having smooth boundary, $\la>0,$ $0 <\mu <N$ and $2_\mu^*=\frac {2N-\mu}{N-2}$ is 
the critical Choquard exponent in view of  Hardy-Littlewood-Sobolev inequality.
In the case of doubly non-local equation,  
Giacomoni et al. \cite{divya}  investigated
$H^s$ versus $C^0$- weighted minimizers of the functional associated to \eqref{dg}.
\par But, to the best of our knowledge, there is no such work regarding Sobolev versus H\"older minimizers for
the  problems involving the
fractional $p$-Laplacian and  critical (or sub-critical) Choquard type nonlinearity.
Also, the tools used in \cite{divya,yangjmaa} to prove this result  can not be adapted for the  general case of $1<p<\infty.$
Therefore,  we establish this result for the problem $(\mathscr P_3)$ considering the degenerate case $(p\geq2).$ 
\par Finally, we show that  if \eqref{mainprob} has a weak sub-solution
and a weak supersolution, then it attains a solution in between the sub-super solutions pair,  
which  also appears as a local minimizer of the associated energy functional to the problem $(\mathscr{P}_3)$ in $W_0^{s,p}(\Om)$
topology.
\par In addition, we also study the aforementioned results, established for \eqref{mainprob}, for the local perturbation of the
	nonlocal Choquard type
	nonlinearity, precisely for the nonlinearity
	$g(x,u)+\left(\DD\int_\Om\frac{F(y,u)}{|x-y|^{\mu}}dy\right)f(x,u),$ where the perturbation term $g(x,u)$ has
	at most critical growth in sense of Sobolev inequality.
	In the end, we apply these results to discuss the multiplicity result when $g(x,u)$ is of concave type.\\\\
\par\noi For doubly non-local equations of type \eqref{mainprob} the main difficulty arises due to the non-Hilbert nature of the solution space and the presence of 
the nonlinear operator $(-\Delta)_s^p,$ as well as, the non-local nonlinearity of Choquard type.
Hence, most of the results and techniques that were used  in  establishing the 
similar kind of regularity results  in the fractional Laplacian or in the local Laplacian set up
(for instance, see \cite{divya,moroz-main,yangjmaa}) are not applicable to \eqref{mainprob}.
Therefore, we need to carry out some extra delicate analysis in our proofs to overcome the stated difficulties.
In \cite{ts},  the regularity of solutions of critical Choquard equations involving the $p$-Laplacian is 
posed as an open problem and in this work, we come up with the answer to it.  In this regard, we would  
like to remark that the regularity results we establish for \eqref{mainprob}  is also valid 
in the  local $p$-Laplacian framework, which are also new to the literature. \\

The plan of the paper is described as follows. In Section \ref{space}, first we recall some  preliminary results regarding fractional Sobolev spaces and state the main results of this article. 
In Section \ref{pmr}, we give proofs of the main results of this article.

\section{Functional settings and statements of the main theorems}\label{space}
In this section, first we collect some known results regarding fractional Sobolev spaces. To study $p$-fractional Sobolev spaces in details we refer to \cite{hit,sarika}.
For $0<s<1$ and $1<p<\infty,$ the fractional Sobolev space is defined as
$$W^{s,p}(\RR^N):= \left\{u\in L^{p}(\RR^N)\bigg| \int_{\RR^N}\int_{\RR^N}\frac{|u(x)-u(y)|^p}{|x-y|^{N+sp}}dxdy<\infty\right\}$$
equipped with the norm
\begin{align*}
\|u\|_{W^{s,p}(\RR^N)}=\|u\|_{L^p(\RR^N)}+ \left(\int_{\RR^N}\int_{\RR^N}
\frac{|u(x)-u(y)|^{p}}{|x-y|^{N+sp}}dxdy \right)^{\frac 1p}.
\end{align*}
We also define
\[ W_0^{s,p}(\Om) = \{u\in W^{s,p}(\RR^N) : u = 0 \;\text{in}\; \RR^N\setminus \Om\}\]
with respect to the norm
\[\|u\|_{s,p}=\left( \int_{\RR^{N}}\int_{\RR^N}\frac{|u(x)- u(y)|^{p}}{|x-y|^{N+sp}}dx
dy\right)^{\frac 1p}= \left(\int_{Q}\frac{|u(x)- u(y)|^{p}}{|x-y|^{N+sp}}dx
dy\right)^{\frac 1p},\]
where  $Q=\RR^{2N}\setminus(\Om^c\times \Om^c).$ Then $W_0^{s,p}(\Om)$ is a reflexive Banach space. Also $W_0^{s,p}(\Om)\hookrightarrow L^q(\RR^N)$ continuously each $q \in [1,p_s^*]$ and $W_0^{s,p}(\Om) \hookrightarrow \hookrightarrow L^q(\Om)$ compactly for each $q \in [1,p_s^*),$ where   $p^*_s= \frac{Np}{N-sp}$ is the Sobolev-type critical exponent. 
The best constant $S_s$ is given below: 
\begin{align}\label{ch27}
S_s= \inf_{u \in W_0^{s,p}(\Om)\setminus \{ 0\}}   \frac{\int_{\RR^N}\int_{\RR^N} \frac{|u(x)-u(y)|^p}{|x-y|^{N+sp}}~dxdy }{\left( \int_{ \Om} |u|^{p^*_s}~dx\right)^{p/p^*_s}}. 
\end{align} The dual of the space $W_0^{s,p}(\Om)$ is denoted by $W^{-s,p'}(\Om)$ with the norm $\|\cdot\|_{-s,p'},$ where $p'=\frac{p}{p-1}$ is  conjugate to $p.$ also, by $\langle\cdot,\cdot\rangle,$ we denote the dual pairing between  $W_0^{s,p}(\Om)$ and $W^{-s,p'}(\Om)$.\\
Let the distance function $d: \overline{\Om} \ra \mathbb{R}_+ $ be defined  by \begin{equation}\label{dist}d(x):= \text{dist}(x,\RR^N\setminus \Om),\; x \in \overline{\Om}.\end{equation}
The weighted H\"older-type spaces are defined as follows:
\begin{align*}
& C^0_d(\overline{\Om}) := \bigg \{ u \in C^0(\overline{\Om}): u/d^s \text{ admits a continuous extension to } \overline{\Om}   \bigg\},\\
& C^{0,\al}_d(\overline{\Om}) := \bigg \{ u \in C^0(\overline{\Om}): u/d^s \text{ admits a } \al \text{ -H\"older continuous extension to } \overline{\Om}   \bigg\}
\end{align*}
equipped with the norms
\begin{align*}
&\|u\|_{C_d^0(\ol{\Om})}:= \|u/d^s\|_{L^\infty(\Om)}, \nonumber\\ &\|u\|_{C_d^{0,\al}(\ol{\Om})}:= \|u\|_{C_d^0(\ol{\Om})}+ \sup_{x, y \in \overline{\Om}, x\not = y } \frac{|u(x)/d^s(x)- u(y)/d^s(y)|}{|x-y|^{\al}},
\end{align*}
respectively.  The embedding $C^{0,\al}_d(\overline{\Om})\hookrightarrow\hookrightarrow C^0_d(\overline{\Om})$ is compact for all $\al\in(0,1).$ \\
\noindent The next lemma states the monotonicity property of the fractional $p$-Laplacian for $p\geq2$.
\begin{lemma}{\rm\cite[Lemma 2.3] {SH}}\label{mon}
	Let $p\ge 2$. There exists $C=C(p)>0$ such that for all $u,v\in W_0^{s,p}(\Om)\cap L^\infty(\Om)$ and all $q\geq 1$
	$$\Big\|(u-v)^{\frac{p+q-1}{p}}\Big\|_{s,p}^p \leq C\,q^{p-1}\,\langle(-\Delta)_s^p u-(-\Delta)_s^p v,(u-v)^q\rangle.$$
\end{lemma}
\noindent The strong Maximum Principle for fractional $p$-Laplacian is given as follows:
\begin{lemma}{\rm {\cite[Lemma 2.3] {mosq}}}{\rm(Strong Maximum Principle)}\label{max}
	Let $u\in W_0^{s,p}(\Om)$ satisfy 
	\begin{equation}
	\left.\begin{array}{rllll}
	(-\Delta)_p^s u
	&\geq 0 \text{\;\;\;\;  weakly in\;}\;
	\Om,\\
	u&\geq 0 \;\;\; \text{ in } \RR^N \setminus \Om.
	\end{array}
	\right\}
	\end{equation}
	Then $u$ has a lower semi-continuous representative in $\Om$, which is either identically $0$ or positive.
\end{lemma}
\noindent Recalling \cite[Theorem 1.1]{fine-bd}, we have the  regularity result for the following problem:
\begin{equation}\label{nl1}
(-\Delta)_p^s u =g\text{ in } \Om, \quad
u=0 \text{ in } \RR^N\setminus \Om. 
\end{equation}
\begin{proposition}\label{regularity}
	Let $2\leq p<\infty$ and $\Om$ be a  bounded domain in $\RR^N$ with $C^{1,1}$ boundary. Let $g\in L^{\infty}(\Om)$ and $h=0.$ Then there exist   $C$ and $\al,$ both positive and depending upon $s,p,\Om$ such that  any weak solution $u\in W_0^{s,p}(\Om)$ of \eqref{nl1} satisfies
	$$\|u\|_{C_d^{0,\al}(\ol{\Om})}\leq C \|g\|_{L^{\infty}(\Om)}^{\frac{1}{p-1}}.$$\end{proposition}
\noindent Now we recall the following  crucial result to handle the nonlocal Choquard type of nonlinearity:
\begin{proposition} {\rm(Hardy–Littlewood–Sobolev inequality) }
	Let $q_1, q_2 > 1$ and
	$0 < \mu < N$ with $1/q_1 + \mu/N + 1/q_2 = 2,$ $g_1 \in L^{q_1}(\RR^N)$ and $g_2 \in L^{q_2}(\RR^N).$ There
	exists a sharp constant $C(q_1,q_2,N, \mu),$ independent of $g, h,$ such that
	\begin{equation}\label{HLS}
	\int_{\RR^N}\int_{{\RR^N}}\frac{g_1(x)g_2(y)}{|x - y|^\mu} dxdy
	\leq C(q_1,q_2,N, \mu)\|g_1\|_{L^{q_1}(\RR^N)}
	\|g_2\|_{L^{q_2}(\RR^N)}.
	\end{equation}
\end{proposition}
\noi Motivated by the inequality \eqref{HLS}, we assume the following hypothesis on the continuous function $f:\Om\times\RR\to\RR$:
\begin{itemize}
	\item[$({\bf H})$] There exists some constant $K_0>0$ such that  for a.e. $x\in\Om$ and for all $t\in\RR:$ $$|f(x,t)|\leq K_0\left(1+|t|^{r-1}\right)$$
	with $1<r\leq p_{\mu,s}^*,$	where $p_{\mu,s}^*:=\frac{(pN -{p{\mu}/2})}{(N - ps)}$ denotes the critical exponent in the sense of Hardy-Littlewood-Sobolev inequality.
\end{itemize}

Observe that $p_{\mu,s}^*=\frac{p_s^*}{{\frac{2N}{2N-\mu}}}< p_s^*.$ 
\begin{definition}{\rm(Weak solution of \eqref{mainprob})}
	$u\in W_0^{s,p}(\Om)$ is said to be weak solution of \eqref{mainprob}, if for all $w\in W_0^{s,p}(\Om)$ 
	{	\begin{align}\label{weak}
		\int_{{\RR^N}}\int_{\RR^N}\frac{|
			u(x)-u(y)|^{p-2}(u(x)-u(y))(w(x)-w(y))}{|
			x-y|^{N+sp}}dxdy=\int_{\Om}\int_{\Om}\frac{F(y,u)f(x,u)}{|x-y|^{\mu}}w(x)~dxdy.
		\end{align}}
\end{definition}
\begin{definition}
	The energy functional $J:W_0^{s,p}(\Om)\to\RR$ associated to the problem \eqref{mainprob} is defined as
	\begin{align}\label{energy}
	J(u)=\frac{1}{p}\|u\|_{{s,p}}-\frac{1}{2}\int_{\Om}\int_{\Om}\frac{F(y,u) F(x,u) }{|x-y|^{\mu}}dxdy.
	\end{align}
\end{definition}
Note that, \eqref{HLS} ensures that the nonlocal terms present in the right-hand side of both \eqref{weak} and \eqref{energy} are well defined.
Now we are in a position to state the main results of this article.
\begin{theorem}\label{reg}
	Let  $2\leq p<\infty$, $s\in(0,1)$ with  $sp<N$ and $\Om$ be a  bounded domain in $\RR^N$ with $C^{1,1}$ boundary. Suppose {$(\bf H)$} holds.
	Then there exists  $\al\in(0,s]$ such that any weak solution
	 $u\in W_0^{s,p}(\Om)$ of \eqref{mainprob} belongs to $\ L^\infty(\RR^N)\cap C^{0,\al}_d(\ol\Om).$
\end{theorem}	
Next, we use Theorem \ref{reg} to have the following result:			

\begin{theorem}\label{thm2}
	Let $2\leq p<\infty,$ $s\in(0,1)$ with  $sp<N$ and $\Om$ be a  bounded domain in $\RR^N$ with $C^{1,1}$ boundary. Suppose {$(\bf H)$} holds. Then for any $w_0\in W_0^{s,p}(\Om)$ the following assertions  are equivalent:
	\begin{itemize}
		\item[$\rm(i)$] there exists $\varrho>0$ such that $J(w_0+ w) \geq J(w_0)$ for all $ w \in W_0^{s,p}(\Om) \cap C^0_d(\overline{\Om})$, $\|w\|_{C_d^0(\ol{\Om})}~\leq \varrho$. 
		\item[$\rm(ii)$] there exists $\delta>0 $ such that $ J(w_0+w) \geq J(w_0) $ for all $  w\in W_0^{s,p}(\Om),$  $\|w\|_{s,p}\leq \delta$. 
		
	\end{itemize}
\end{theorem}	

\begin{definition}
	Let  $u\in { W}^{s,p}(\RR^N).$ Then\begin{itemize}
		\item[$(i)$] $u$ is a super-solution of \eqref{mainprob}, if  we have $u\geq0$ a.e. $x\in\Om^c$ and for all $v\in{W}^{s,p}_0(\Om)$ with $v\geq0$ a.e. in $\Om$,

		$$\int_{{\RR^N}}\int_{\RR^N}\frac{|
			u(x)-u(y)|^{p-2}(u(x)-u(y))(v(x)-v(y))}{|
			x-y|^{N+sp}}dxdy\geq\int_{\Om}\int_{\Om}\frac{F(y,u)f(x,u)}{|x-y|^{\mu}}v(x);$$
		\item[$(ii)$] $u$ is a sub-solution of \eqref{mainprob}, if  we have $u\leq0$ a.e. $x\in\Om^c$ and for all $v\in{W}^{s,p}_0(\Om)$ with $v\geq0$ a.e. in $\Om$
		$$\int_{{\RR^N}}\int_{\RR^N}\frac{|
			u(x)-u(y)|^{p-2}(u(x)-u(y))(v(x)-v(y))}{|
			x-y|^{N+sp}}dxdy\leq\int_{\Om}\int_{\Om}\frac{F(y,u)f(x,u)}{|x-y|^{\mu}}v(x).$$
	\end{itemize}
\end{definition}
Now we  discuss the following result where Theorem  \ref{thm2} plays an important role to  ensure that, if problem \eqref{mainprob} has a weak subsolution
and a weak supersolution, then it achieves a solution which is also a local minimizer of $J$ in  $W_0^{s,p}(\Om)$. 
\begin{theorem}\label{thmch4}Let $2< p<\infty,$ $s\in(0,1)$ with  $sp<N$ and $\Om$ be a bounded domain in $\RR^N$ with $C^{1,1}$ boundary. Let  $(\bf{H})$ hold
	and $f(x,\cdot)$  be non decreasing function in $\RR$ for all $x \in \Om$. Suppose $\underline{v}, \overline{v} \in W_0^{s,p}(\Om)$ are a weak subsolution and a
	weak supersolution, respectively to \eqref{mainprob}, which are not solutions, such that $\underline v\leq \overline v$. Then, there exists  a solution $ v_0\in W_0^{s,p}(\Om)$ to \eqref{mainprob} 
	such that $\underline{v} \leq v_0 \leq \overline{v}$ a.e in $\Om$ and $v_0$ is a local minimizer of $J$ in $W_0^{s,p}(\Om)$. 
\end{theorem}
	\noi Next, we consider the following perturbation of \eqref{mainprob}:		
	\begin{equation}\label{mainprob1}
	\;\;\;	\left.\begin{array}{rl}	
	(-\Delta)_p^su&={ g(x,u)+}\left(\DD\int_\Om\frac{F(y,u)}{|x-y|^{\mu}}dy\right)f(x,u),\hspace{5mm}x\in \Om,\\\\
	u&=0,\hspace{40mm}x\in \RR^N\setminus \Om,
	\end{array}
	\right\}
	\end{equation}
	where $\Om\subset\RR^N$ is a bounded 
	domain with $C^{1,1}$ boundary, $1<p<\infty$ and $0<s<1$ such that $sp<N$. 
	We assume the following hypothesis on the Carath\'eodory function $g:\Om\times\RR\to\RR$:
	\begin{itemize}
		\item[$({\bf H'})$] There exists some constant $C_0>0$ such that  for a.e. $x\in\Om$ 
		and for all $t\in\RR:$ $$|g(x,t)|\leq C_0\left(1+|t|^{q-1}\right),$$ where $1<q\leq p_{s}^*$.
	\end{itemize}
	
	\begin{definition}{\rm(Weak solution of \eqref{mainprob1})}
		$u\in W_0^{s,p}(\Om)$ is said to be weak solution of \eqref{mainprob1}, if for all $w\in W_0^{s,p}(\Om),$ it holds that 
			\begin{align}\label{weak1}
			&\int_{{\RR^N}}\int_{\RR^N}\frac{|
				u(x)-u(y)|^{p-2}(u(x)-u(y))(w(x)-w(y))}{|
				x-y|^{N+sp}}dxdy\n&\qquad=\int_{\Om} g(x,u)w dx+\int_{\Om}\int_{\Om}\frac{F(y,u)f(x,u)}{|x-y|^{\mu}}w(x)~dxdy.
			\end{align}
	\end{definition}
	The weak solution of \eqref{mainprob1} is characterized as the critical point of the 
	associated energy functional $J_1:W_0^{s,p}(\Om)\to\RR,$ defined as follows:
	\begin{align}\label{energy1}
	J_1(u)=\frac{1}{p}\|u\|_{{s,p}}-\int_{\Om}G(x,u)dx-\frac{1}{2}\int_{\Om}\int_{\Om}\frac{F(y,u) F(x,u) }{|x-y|^{\mu}}dxdy,
	\end{align} where $G(x,t):=\int_{0}^tg(x,\tau)dx$ is the primitive of $g$.
	Analogous to \eqref{mainprob}, we have the following result for \eqref{mainprob1}.
	\begin{theorem}\label{pert}
		Let the assumptions in Theorem \ref{reg}  and  {$(\bf H')$} hold.
		Then, there exists  $\al\in(0,s]$ such that any weak solution $u\in W_0^{s,p}(\Om)$ of \eqref{mainprob1}
		belongs to $ L^\infty(\RR^N)\cap C^{0,\al}_d(\ol\Om)$. 
		Also, under the assumptions in Theorem \ref{thm2} and  {$(\bf H')$},
		the assertions of Theorem \ref{thm2} hold for the functional $J_1$.
	\end{theorem}
	\begin{definition}
		Let $u\in  W^{s,p}(\RR^N).$ Then\begin{itemize}
			\item[$(i)$] $u$ is a supersolution of \eqref{mainprob1},if  we have $u\geq0$ a.e. $x\in\Om^c$ and for all $v\in{W}^{s,p}_0(\Om)$ with $v\geq0$ a.e. in $\Om$
			\begin{align*}&\int_{{\RR^N}}\int_{\RR^N}\frac{|
				u(x)-u(y)|^{p-2}(u(x)-u(y))(v(x)-v(y))}{|
				x-y|^{N+sp}}dxdy\\&\qquad\geq\int_{\Om}g(x,u)v(x)dx+\int_{\Om}\int_{\Om}\frac{F(y,u)f(x,u)}{|x-y|^{\mu}}v(x) dx dy;\end{align*}
			\item[$(ii)$] $u$ is a subsolution of \eqref{mainprob1}, if  we have $u\leq0$ a.e. $x\in\Om^c$ and for all $v\in{W}^{s,p}_0(\Om)$ with $v\geq0$ a.e. in $\Om$
		\begin{align*}&\int_{{\RR^N}}\int_{\RR^N}\frac{|
				u(x)-u(y)|^{p-2}(u(x)-u(y))(v(x)-v(y))}{|
				x-y|^{N+sp}}dxdy\\&\qquad\leq\int_{\Om}g(x,u)v(x)dx+\int_{\Om}\int_{\Om}\frac{F(y,u)f(x,u)}{|x-y|^{\mu}}v(x) dx dy.\end{align*}
		\end{itemize}
	\end{definition}
	\begin{theorem}\label{thmch4.1}
		Let the assumptions of Theorem \ref{thmch4} and $(\bf{H}')$ hold.
		Also, let $g(x,\cdot)$ be non decreasing function in $\RR$ for all $x \in \Om$. Suppose $\underline{w}, \overline{w} \in W_0^{s,p}(\Om)$ are a weak subsolution and a
		weak supersolution, respectively to \eqref{mainprob1}, which are not solutions, such that  $\underline w\leq \overline w$. Then, there exists  a solution $ w_0\in W_0^{s,p}(\Om)$ to \eqref{mainprob1} 
		such that $\underline{w} \leq w_0 \leq \overline{w}$ a.e in $\Om$ and $w_0$ is a local minimizer of $J_1$ in $W_0^{s,p}(\Om)$. 
	\end{theorem}

\section{Proofs of the main results}\label{pmr} In this section, we consider $C$ to be a generic positive constant which may vary from line to line.
In order to prove the main theorems of this article, we recall  some useful  inequalities.
\begin{lemma}{\rm \cite[Lemma C.2]{cheeger}}\label{lem1}
	Let $1<p<\infty$ and $\ba\geq1.$	For every $a,b.m\geq0$ there holds  $$|a-b|^{p-2}(a-b)(a_m^{\ba}-b_m^{\ba})
	\geq \frac{{\ba}p^p}{({\ba}+p-1)^p}\left|a_m^{\frac{{\ba}+p-1}{p}}-b_m^{\frac{{\ba}+p-1}{p}}\right|^p,$$
	where we set $a_m=\min\{a,m\}$ and $b_m=\min\{b,m\}.$
\end{lemma}
\begin{lemma}{\rm\cite [Lemma A.1] {second-eigenvalue}}\label{lem2}
	Let $1<p<\infty$ and $f:\RR\to\RR$ be a  convex function. 
	\begin{align*}
	&|a - b|^{p-2} (a - b)\left[
	A~|f'(a)|^{p-2} f'(a) - B ~|f'(b)|^{p-2} f'(b)\right]\n&\qquad\geq |f(a) - f(b)|^{p-2} (f(a)-f(b)) (A - B),
	\end{align*}
	for every $a, b \in \RR$ and every $A,B \geq 0.$
\end{lemma}
\noi In the next lemma, following the approach as in (\cite{second-eigenvalue}, Theorem 3.1), we derive $a~priori$ bound on the weak solution of \eqref{mainprob}. 
\begin{lemma}{\rm(Global $L^\infty$- bound)}\label{L-infty} Let the assumptions in Theorem \ref{reg} hold.
	Then  any weak solution $u\in W_0^{s,p}(\Om)$ of \eqref{mainprob} belongs to  $ L^\infty(\RR^N).$
\end{lemma}

\begin{proof}  From the given assumption on $\mu$, we have $p_{\mu,s}^*>p.$ We take  $\phi=\psi|h_\e'(u)|^{p-2}h_\e'(u)$ as the test function
	in \eqref{weak}, where $\psi\in C_c^\infty(\Om), \psi>0$ and 
	for every $0 < \e <<1, $ we define the smooth convex Lipschitz function $$h_\e(t)=(\e^2+t^2)^{\frac{1}{2}}.$$ In addition, by choosing
	$$a = u(x),~ b = u(y),~ A = \psi(x) \text{~~ and~~} B = \psi(y)$$ in Lemma \ref{lem2}, we obtain
	\begin{align}\label{1}
	&\int_{{\RR^N}}\int_{\RR^N}\frac{|
		h_\e(u(x))-h_\e(u(y))|^{p-2}(h_\e(u(x))-h_\e(u(y)))(\psi(x)-\psi(y))}{|
		x-y|^{N+sp}}dxdy\n&~~~~~~~~~~~~~~~~~~~~~~~~\leq\int_{\Om}\int_{\Om}\frac{|F(y,u)|~|f(x,u)|}{|x-y|^{\mu}}~ |h_\e'(u(x))|^{p-1}\psi(x)\;dxdy.
	\end{align}
	Since $h_\e(t)$ converges to $h(t):=|t|$ as $\e\to 0^+$ and $|h_\e'(t)|\leq 1$, passing to the limit and using Fatou's lemma in \eqref{1}, we get
	\begin{align}\label{2}
	\int_{{\RR^N}}\int_{\RR^N}\frac{\Big|
		|u(x)|-|u(y)|\Big|^{p-2}\Big(|u(x)|-|u(y)|\Big)\Big(\psi(x)-\psi(y)\Big)}{|
		x-y|^{N+sp}}dxdy\leq\int_{\Om}\int_{\Om}\frac{|F(y,u)|\;|f(x,u)|}{|x-y|^{\mu}}\psi(x)\;dxdy
	\end{align} for every positive $\psi\in C_c^\infty(\Om).$ By density, \eqref{2} holds true for $0\leq\psi\in W_0^{s,p}(\Om).$ 
	Next, we define $$u_l=\min\{l,|u(x)|\}.$$ Clearly $u_l\in W_0^{s,p}(\Om).$ For $k\geq1,$ let us set
	$$\beta:=kp-p+1.$$ So $\beta>1.$ In \eqref{2}, choosing $\psi=u_l^{\ba },$ and using Lemma \ref{lem1}, we obtain
	\begin{align}\label{3}
	\frac{\beta p^p}{(\ba+p-1)^p}\int_{\RR^N}\int_{\RR^N}\frac{\left|(u_l(x))^{\frac{\ba+p-1}{p}}-(u_l(y))^{\frac{\ba+p-1}{p}}\right|^p}{|x-y|^{N+sp}}dxdy\leq\int_{\Om}\int_{\Om}\frac{|F(y,u)|\;|f(x,u)|}{|x-y|^{\mu}}(u_l(x))^{\ba}\;dxdy.
	\end{align}
	By observing  that   $$\frac{1}{\beta}\left({\frac{\ba+p-1}{p}}\right)^p\leq \left({\frac{\ba+p-1}{p}}\right)^{p-1}, 
	\text{~~for large~} \ba$$ and using the continuous embedding $W_0^{s,p}(\Om)\hookrightarrow L^{p_s^*}(\Om)$, from \eqref{3} we get
	\begin{align}\label{4}
	\left\|u_l^{k}\right\|^p_{L^{p_s^*}(\Om)}\leq \frac {(k)^{p-1}}{ S_s^p}\int_{\Om}\int_{\Om}\frac{|F(y,u)|\;|f(x,u)|}{|x-y|^{\mu}}(u_l(x))^{\ba}\;dxdy,
	\end{align}
	where we have used the relation $k=\frac{\ba+p-1}{p}$ and $S_s$ is as defined in \eqref{ch27}.
	Now we will estimate the right-hand side  of \eqref{4}. Using $(\bf H),$  Hardy-Littlewood-Sobolev inequality and the fact $u_l\leq |u|$ and making use of the inequalities $ (x_1+x_2)^{\gamma}\leq x_1^{\gamma}+x_2^{\gamma},~0<\gamma<1, ~x_1,x_2\geq 0,$ and $ (x_1+x_2)^{\gamma}\leq 2^{\gamma-1} (x_1^{\gamma}+x_2^{\gamma}),~\gamma>1,~ x_1,x_2\geq 0,$ we deduce
		\begin{align}\label{5}
	&\int_{\Om}\int_{\Om}\frac{|F(y,u)|\;|f(x,u)|}{|x-y|^{\mu}}(u_l(x))^{\ba}~dxdy\n
	&\leq C\|F(\cdot,u(\cdot))\|_{L^{\frac{2N}{2N-\mu}}(\Om)}\left(\int_\Om\left(|f(x,u)|~|u_l(x)|^\beta\right)^{\frac{2N}{2N-\mu}}\right)^{\frac{2N-\mu}{2N}}\n
	&\leq C\left(\|u\|_{L^{\frac{2N}{2N-\mu}}(\Om)}+\left\||u|^{p_{\mu,s}^*}\right\|_{L^{{\frac{2N}{2N-\mu}}}(\Om)}\right)\;\left(\int_\Om|u_l|^{\frac{2N\beta}{2N-\mu}} dx+\int_\Om\left(|u|^{p_{\mu,s}^*-2} |u\;u_l^{\ba}|\right)^{\frac{2N}{2N-\mu}}dx\right)^{\frac{2N-\mu}{2N}}\n
	&=C\left(\|u\|_{L^{\frac{p_s^*}{p_{\mu,s}^*}}(\Om)}+\left\|u\right\|_{L^{{{p_s^*}}}(\Om)}^{p_{\mu,s}^*}\right)\Bigg[\int_{\Om\cap\{|u|<\Lambda\}}|u_l|^{\beta\frac{p_s^*}{p_{\mu,s}^*}} dx+\int_{\Om\cap\{|u|\geq\Lambda\}}|u_l|^{\beta\frac{p_s^*}{p_{\mu,s}^*}} dx\n&\qquad\qquad\qquad+\int_{\Om\cap\{|u|<\Lambda\}}\left(|u|^{p_{\mu,s}^*-2} |u\;u_l^{\ba}| \right)^{\frac{p_s^*}{p_{\mu,s}^*}}dx+\int_{\Om\cap\{|u|\geq\Lambda\}}\left(|u|^{p_{\mu,s}^*-2} |u\;u_l^{\ba}|\right)^{\frac{p_s^*}{p_{\mu,s}^*}}dx
	\Bigg]^{\frac{p_{\mu,s}^*}{p_s^*}}\n
	&\leq\tilde{ C}\Bigg[\left(\int_{\Om\cap\{|u|<\Lambda\}}|u|^{\beta\frac{p_s^*}{p_{\mu,s}^*}}dx\right)^{\frac{p_{\mu,s}^*}{p_s^*}}+\left(\int_{\Om\cap\{|u|\geq\Lambda\}}\left(|u|^{p_{\mu,s}^*+\beta-1}\right)^{\frac{p_s^*}{p_{\mu,s}^*}}dx\right)^{\frac{p_{\mu,s}^*}{p_s^*}}\n
	&
	\qquad+\left(\int_{\Om\cap\{|u|<\Lambda\}}\left(|u|^{p_{\mu,s}^*+\beta-1} \right)^{\frac{p_s^*}{p_{\mu,s}^*}}dx\right)^{\frac{p_{\mu,s}^*}{p_s^*}}+\left(\int_{\Om\cap\{|u|\geq\Lambda\}}\left(|u|^{p_{\mu,s}^*+\beta-1} \right)^{\frac{p_s^*}{p_{\mu,s}^*}}dx\right)^{\frac{p_{\mu,s}^*}{p_s^*}}\Bigg]\n
	&\leq\tilde{ C}\Bigg[|\Om|^{\frac{p_{\mu,s}^*-1}{p_{\mu,s}^*+\beta-1}.\frac{p_{\mu,s}^*}{p_s^*}}\left\{\left(\int_{\Om\cap\{|u|<\Lambda\}}|u|^{\frac{p_s^*}{p_{\mu,s}^*}{(p_{\mu,s}^*+\beta-1)}}dx\right)^{\frac{\beta}{p_{\mu,s}^*+\beta-1}}\right\}^{\frac{p_{\mu,s}^*}{p_s^*}}\n
	&
	\qquad+\left(\int_{\Om\cap\{|u|<\Lambda\}}\left(|u|^{p_{\mu,s}^*+\beta-1} \right)^{\frac{p_s^*}{p_{\mu,s}^*}}dx\right)^{\frac{p_{\mu,s}^*}{p_s^*}}+2\left(\int_{\Om\cap\{|u|\geq\Lambda\}}\left(|u|^{p_{\mu,s}^*-p}\; |u|^{p+\beta-1}\right)^{\frac{p_s^*}{p_{\mu,s}^*}}dx\right)^{\frac{p_{\mu,s}^*}{p_s^*}}\Bigg]
	\end{align}
	where $\Lambda>1$ will be chosen later and $\tilde{C}=C\left(\|u\|_{L^{\frac{p_s^*}{p_{\mu,s}^*}}(\Om)}+\left\|u\right\|_{L^{{{p_s^*}}}(\Om)}^{p_{\mu,s}^*}\right).$
	 Next we proceed by  adapting the idea of the proof of \cite[ Theorem 4.1]{pat}. Now two cases arise for the first integration expression in the right-hand side of \eqref{5}.
	{\begin{align*}
		&\text{Either  } \left(\int_{\Om\cap\{|u|<\Lambda\}}|u|^{\frac{p_s^*}{p_{\mu,s}^*}{(p_{\mu,s}^*+\beta-1)}}dx\right)^{\frac{\beta}{p_{\mu,s}^*+\beta-1}}\leq 1 \text{  or  }>1, \text{\;\; that is,}\n
		&\text{ either  }\left(\int_{\Om\cap\{|u|<\Lambda\}}|u|^{\frac{p_s^*}{p_{\mu,s}^*}{(p_{\mu,s}^*+\beta-1)}}dx\right)^{\frac{\beta}{p_{\mu,s}^*+\beta-1}}\leq 1\\&\qquad\qquad\qquad\text{  or  }\\ &\left(\int_{\Om\cap\{|u|<\Lambda\}}|u|^{\frac{p_s^*}{p_{\mu,s}^*}{(p_{\mu,s}^*+\beta-1)}}dx\right)^{\frac{\beta}{p_{\mu,s}^*+\beta-1}}<\int_{\Om\cap\{|u|<\Lambda\}}|u|^{\frac{p_s^*}{p_{\mu,s}^*}{(p_{\mu,s}^*+\beta-1)}}dx.
		\end{align*}}
	The above implies 
	\begin{align}\label{n1}
	\left(\int_{\Om\cap\{|u|<\Lambda\}}|u|^{\frac{p_s^*}{p_{\mu,s}^*}(p_{\mu,s}^*+\beta-1)}dx\right)^{\frac{\beta}{p_{\mu,s}^*+\beta-1}}\leq	1+\int_{\Om\cap\{|u|<\Lambda\}}\left(|u|^{p_{\mu,s}^*+\beta-1}\right)^{\frac{p_s^*}{p_{\mu,s}^*}}dx.
	\end{align}
	Plugging \eqref{n1} into \eqref{5}, we have 
	{	\begin{align}\label{n2}
		&\int_{\Om}\int_{\Om}\frac{|F(y,u)|\;|f(x,u)|}{|x-y|^{\mu}}(u_l(x))^{\ba}~dxdy\n		&\leq \tilde{ C}\Bigg[C'\left(1+\int_{\Om\cap\{|u|<\Lambda\}}\left(|u|^{p_{\mu,s}^*+\beta-1}\right)^{\frac{p_s^*}{p_{\mu,s}^*}}dx\right)^{\frac{p_{\mu,s}^*}{p_s^*}}+\Lambda^{p_{\mu,s}^*-p}\left(\int_{\Om\cap\{|u|<\Lambda\}} \left(|u|^{p+\beta-1}\right)^{\frac{p_s^*}{p_{\mu,s}^*}}dx\right)^{\frac{p_{\mu,s}^*}{p_s^*}}\n&\qquad\qquad+2\left(\int_{\Om\cap\{|u|\geq\Lambda\}}\left(|u|^{p_{\mu,s}^*-p}\; |u|^{p+\beta-1}\right)^{\frac{p_s^*}{p_{\mu,s}^*}}dx\right)^{\frac{p_{\mu,s}^*}{p_s^*}}\Bigg]\n&= \tilde{ C}\Bigg[C'\left(1+\int_{\Om\cap\{|u|<1\}}\left(|u|^{p_{\mu,s}^*+\beta-1}\right)^{\frac{p_s^*}{p_{\mu,s}^*}}dx+\int_{\Om\cap\{1<|u|<\Lambda\}}\left(|u|^{p_{\mu,s}^*+\beta-1}\right)^{\frac{p_s^*}{p_{\mu,s}^*}}dx\right)^{\frac{p_{\mu,s}^*}{p_s^*}}+\Lambda^{p_{\mu,s}^*-p}\;\|u\|_{L^{\frac{kpp_s^*}{p_{\mu,s}^*}}(\Om)}^{kp}\n&\qquad\qquad+2\left(\int_{\Om\cap\{|u|\geq\Lambda\}}\left(|u|^{p_{\mu,s}^*-p}\; |u|^{kp}\right)^{\frac{p_s^*}{p_{\mu,s}^*}}dx\right)^{\frac{p_{\mu,s}^*}{p_s^*}}\Bigg]\n
		&\leq \tilde{ C}\Bigg[C'\left\{1+\left(\int_{\Om\cap\{|u|<1\}}\left(|u|^{p_{\mu,s}^*+\beta-1}\right)^{\frac{p_s^*}{p_{\mu,s}^*}}dx\right)^{\frac{p_{\mu,s}^*}{p_s^*}}+\left(\int_{\Om\cap\{1<|u|<\Lambda\}}\left(|u|^{p_{\mu,s}^*+\beta-1}\right)^{\frac{p_s^*}{p_{\mu,s}^*}}dx\right)^{\frac{p_{\mu,s}^*}{p_s^*}}\right\}\n&\qquad\qquad+\Lambda^{p_{\mu,s}^*-p}\;\|u\|_{L^{\frac{kpp_s^*}{p_{\mu,s}^*}}(\Om)}^{kp}+2\left(\int_{\Om\cap\{|u|\geq\Lambda\}}\left(|u|^{p_{\mu,s}^*-p}\; |u|^{kp}\right)^{\frac{p_s^*}{p_{\mu,s}^*}}dx\right)^{\frac{p_{\mu,s}^*}{p_s^*}}\Bigg]\n
		&\leq \tilde{ C}\Bigg[C'\left\{1+\left(\int_{\Om\cap\{|u|<1\}}\left(|u|^{p+\beta-1}\right)^{\frac{p_s^*}{p_{\mu,s}^*}}dx\right)^{\frac{p_{\mu,s}^*}{p_s^*}}+\Lambda^{p_{\mu,s}^*-p}\;\|u\|_{L^{\frac{kpp_s^*}{p_{\mu,s}^*}}(\Om)}^{kp}\right\}+\Lambda^{p_{\mu,s}^*-p}\;\|u\|_{L^{\frac{kpp_s^*}{p_{\mu,s}^*}}(\Om)}^{kp}\n&\qquad\qquad+2\left(\int_{\Om\cap\{|u|\geq\Lambda\}}\left(|u|^{p_{\mu,s}^*-p}\; |u|^{kp}\right)^{\frac{p_s^*}{p_{\mu,s}^*}}dx\right)^{\frac{p_{\mu,s}^*}{p_s^*}}\Bigg]\n	&= \tilde{ C}\Bigg[C'\left\{1+\|u\|_{L^{\frac{kpp_s^*}{p_{\mu,s}^*}}(\Om)}^{kp}+\Lambda^{p_{\mu,s}^*-p}\;\|u\|_{L^{\frac{kpp_s^*}{p_{\mu,s}^*}}(\Om)}^{kp}\right\}+\Lambda^{p_{\mu,s}^*-p}\;\|u\|_{L^{\frac{kpp_s^*}{p_{\mu,s}^*}}(\Om)}^{kp}\n&\qquad\qquad\qquad+2\left(\int_{\Om\cap\{|u|\geq\Lambda\}}\left(|u|^{p_{\mu,s}^*-p}\; |u|^{kp}\right)^{\frac{p_s^*}{p_{\mu,s}^*}}dx\right)^{\frac{p_{\mu,s}^*}{p_s^*}}\Bigg]\n
		&\leq C''\Bigg[1+\Lambda^{p_{\mu,s}^*-p}\;\|u\|_{L^{\frac{kpp_s^*}{p_{\mu,s}^*}}(\Om)}^{kp}+\left(\int_{\Om\cap\{|u|\geq\Lambda\}}\left(|u|^{p_{\mu,s}^*-p}\; |u|^{kp}\right)^{\frac{p_s^*}{p_{\mu,s}^*}}dx\right)^{\frac{p_{\mu,s}^*}{p_s^*}}\Bigg],
		\end{align}} where $C',\, C''>1$ are some positive constants that do not depend on $k,\beta.$
	Again by plugging  \eqref{n2} into \eqref{4} and applying Fatou's lemma, we get
	\begin{align}\label{6}
	\|u\|_{L^{k p_s^*}(\Om)}^{k p}
	\leq C'' \frac {k^{p-1}}{ (S_s)^p}\Bigg[1+ \Lambda^{p_{\mu,s}^*-p}\;\|u\|_{L^{\frac{kpp_s^*}{p_{\mu,s}^*}}(\Om)}^{kp}+\left(\int_{\Om\cap\{|u|\geq\Lambda\}}\left(|u|^{p_{\mu,s}^*-p}\; |u|^{kp}\right)^{\frac{p_s^*}{p_{\mu,s}^*}}dx\right)^{\frac{p_{\mu,s}^*}{p_s^*}}\Bigg].
	\end{align}
	Now using H\"older inequality, we obtain
	\begin{align}\label{7}
	&\left(\int_{\Om\cap\{|u|\geq\Lambda\}}\left(|u|^{p_{\mu,s}^*-p} \left(|u|^{kp}\right)\right)^{\frac{p_s^*}{p_{\mu,s}^*}}dx\right)^{\frac{p_{\mu,s}^*}{p_s^*}}\n
	&\leq C \left(\int_{\Om\cap\{|u|\geq\Lambda\}}\Big(|u|^{(p_{\mu,s}^*-p)\frac{p_s^*}{p_{\mu,s}^*}} \Big)^{\frac{p_{\mu,s}^*}{p_{\mu,s}^*-p}}dx\right)^{\frac{p_{\mu,s}^*}{p_s^*}.\frac{p_{\mu,s}^*-p}{p_{\mu,s}^*}}\left(\int_{\Om\cap\{|u|\geq\Lambda\}}\Big(|u|^{k p.\frac{p_s^*}{p_{\mu,s}^*}} \Big)^{\frac{p_{\mu,s}^*}{p}}dx\right)^{\frac{p_{\mu,s}^*}{p_s^*}.\frac{p}{p_{\mu,s}^*}}\n
	&\leq C\left(\int_{\Om\cap\{|u|\geq\Lambda\}}|u|^{p_s^*} dx\right)^{\frac{p_{\mu,s}^*-p}{p_s^*}}\left(\int_{\Om}|u|^{k p_s^*} dx\right)^{\frac{p}{p_s^*}}=C(\Lambda)
	\|u\|_{L^{k p_s^*}(\Om)}^{k p}	\end{align}
	Combining \eqref{6} and \eqref{7}, we have
	\begin{align}\label{8}
	\|u\|_{L^{k p_s^*}(\Om)}^{k p}\leq {C''} \frac {k^{p-1}}{ (S_s)^p}\left[1+\Lambda^{p_{\mu,s}^*-p} \|u\|_{L^{\frac{k pp_s^*}{p_{\mu,s}^*}}(\Om)}^{kp}
	+ C(\Lambda)\|u\|_{L^{k p_s^*}(\Om)}^{k p}
	\right].
	\end{align} 
	Now  by Lebesgue dominated convergence theorem in \eqref{7}, we  choose $\Lambda>1$ large enough so that $C(\Lambda)$ is appropriately small and consequently $\DD C(\Lambda)< \frac{ (S_s)^p} {2{C''}(k)^{p-1}}.$
	Therefore, by employing the last inequality in \eqref{8}, it follows that
	\begin{align}\label{9.0}
	\|u\|_{L^{k p_s^*}(\Om)}\leq \left(\hat C^{\frac{1}{k}}\right)^{\frac{1}{p}} (k^{\frac{1}{k}})^{\frac{p-1}{p}}\left(1+\left(\int_{\Om}|u|^{\frac{k pp_s^*}{p_{\mu,s}^*}} dx\right)^{\frac{p_{\mu,s}^*}{p_s^*}}\right)^{\frac{1}{kp}},      
	\end{align}
	where $\DD \hat C= \frac {2{C''}~\Lambda^{p_{\mu,s}^*-p}}{ (S_s)^p}>1.$ Now we use bootstrap argument on \eqref{9.0}. For that, we argue as follows:\\
	If there exists a sequence  $k_n\to\infty$ as $n\to\infty$ such that $$\int_{\Om}|u|^{\frac{k_n pp_s^*}{p_{\mu,s}^*}} dx\leq1,$$ then from \eqref{9.0}, it immediately follows that 
	$$\|u\|_{L^{\infty}(\Om)}\leq1.$$  
	If there is no such sequence satisfying the above condition, then there exists $k_0>0$ such that
	$$\int_{\Om}|u|^{\frac{k pp_s^*}{p_{\mu,s}^*}} dx>1,\text{\; for\;all\; } k\geq k_0.$$	Then from \eqref{9.0}, we infer that 
	\begin{align}\label{9}
	\|u\|_{L^{k p_s^*}(\Om)}\leq \left(C_*^{\frac{1}{k}}\right)^{\frac{1}{p}} (k^{\frac{1}{k}})^{\frac{p-1}{p}} \|u\|_{L^{\frac{k pp_s^*}{p_{\mu,s}^*}}(\Om)}, \text{  for\;all } k\geq k_0,
	\end{align} where $C_*=2\hat C>1.$ 
	Choose  $k=k_1:=k_0\frac{p_{\mu,s}^*}{p}>1$ as the first iteration.   Thus,  \eqref{9} yields that
	\begin{align}\label{it1}
	\|u\|_{L^{k{_1} p_s^*}(\Om)}\leq \left(C_*^{\frac{1}{k{_1}}}\right)^{\frac{1}{p}} (k_{1}^{\frac{1}{k_{1}}})^{\frac{p-1}{p}}\|u\|_{L^{k_0 p_s^*}(\Om)}.
	\end{align}
	Again by taking $k=k_2:=k_1 \frac{p_{\mu,s}^*}{p}$ as the second iteration in \eqref{9} and then employing   \eqref{it1} in it, we get
	\begin{align}\label{it2}
	\|u\|_{L^{k_2 p_s^*}(\Om)}&\leq \left(C_*^{\frac{1}{k_2}}\right)^{\frac{1}{p}} \left[(k_2)^{\frac{1}{k_2}}\right]^{\frac{p-1}{p}}\|u\|_{L^{ k_1 p_s^*}(\Om)}\n
	&\leq \left(C_*^{\frac{1}{k_1}+\frac{1}{k_2}}\right)^{\frac{1}{p}} \left[(k_1)^{\frac{1}{k_1}}.(k_2)^{\frac{1}{k_2}}\right]^{\frac{p-1}{p}}\|u\|_{L^{k_0  p_s^*}(\Om)}.
	\end{align}
	In this fashion,  taking $k=k_n:=k_{n-1}\frac{p_{\mu,s}^*}{p}$ as the $n^{th}$ iteration and iterating for $n$ times , we obtain
	\begin{align}\label{itn}
	\|u\|_{L^{k_n p_s^*}(\Om)}&\leq \left(C_*^{\frac{1}{k_n}}\right)^{\frac{1}{p}} \left[(k_n)^{\frac{1}{k_n}}\right]^{\frac{p-1}{p}}\|u\|_{L^{ k_{n-1} p_s^*}(\Om)}\n
	&\leq \left(C_*^{\DD\frac{1}{k_1}+\frac{1}{k_2}\cdots +\frac{1}{k_n}}\right)^{\frac{1}{p}} \left[(k_1)^{\DD\frac{1}{k_1}}.(k_2)^{\DD\frac{1}{k_2}}\cdots (k_n)^{\DD\frac{1}{k_n}}\right]^{\frac{p-1}{p}}\|u\|_{L^{k_0  p_s^*}(\Om)}\n
	&=\left(C_*^{\DD\sum_{j=1}^{n}{\frac{1}{k_j}}}\right)^{\frac{1}{p}}\left(\prod_{j=1}^{n}\left(k_j^{\sqrt{1/{k_j}}}\right)^{\sqrt{1/{k_j}}}\right)^{\frac{p-1}{p}}\|u\|_{L^{ k_0 p_s^*}(\Om)},
	\end{align}
	where $k_j=\left(\frac{p_{\mu,s}^*}{p}\right)^j.$ Since $\frac{p_{\mu,s}^*}{p}>1,$ we have $k_j^{\DD\sqrt{1/{k_j}}}>1$ for all $j\in\mathbb N$ and
	$$\lim_{j\to\infty}k_j^{\DD\sqrt{1/{k_j}}}=1.$$
	Hence, it follows that there exists a constant $C^*>1,$ independent of $n,$ such that $k_j^{\DD\sqrt{1/{k_j}}}<C^*$ and thus, \eqref{itn} gives
	\begin{align}\label{10}
	\|u\|_{L^{k_n p_s^*}(\Om)}&\leq \left(C_*^{\DD\sum_{j=1}^{n}{\frac{1}{k_j}}}\right)^{\frac{1}{p}}\left({C^*}^{\DD\sum_{j=1}^{n}\sqrt{1/{k_j}}}\right)^{\frac{p-1}{p}}\|u\|_{L^{k_0  p_s^*}(\Om)}.
	\end{align}
	As limit $n\to\infty,$ the sum of the following geometric series are given as: $$\sum_{j=1}^{\infty}{\frac{1}{k_j}}=\sum_{j=1}^{n}\left({\frac{p}{p_{\mu,s}^*}}\right)^j=\frac{p/p_{\mu,s}^*}{1-p/p_{\mu,s}^*}=\frac{p}{p_{\mu,s}^*-p}$$ and
	$$\sum_{j=1}^{\infty}{\frac{1}{\sqrt{k_j}}}=\sum_{j=1}^{n}{\left(\sqrt{\frac{p}{p_{\mu,s}^*}}\right)^{j}}=\frac{\sqrt p}{\sqrt{ p_{\mu,s}^*}-\sqrt p}.$$
	Thus, from the last two relations and \eqref{10}, we get that
	\begin{align}\label{contr}
	\|u\|_{L^{\nu_n}(\Om)}\leq \left(C_*\right)^{\DD\frac{1}{p_{\mu,s}^*-p}}\left(C^*\right)^{\DD\frac{p-1}{\sqrt {p}(\sqrt{p_{\mu,s}^*}-\sqrt {p})}}\|u\|_{L^{k_0  p_s^*}(\Om)},
	\end{align} where $\nu_n:=k_n p_s^*.$ Note that, $\nu_n\to\infty$ as $n\to\infty.$ Therefore, we claim that  
	\begin{align}\label{claim}
	u\in L^\infty(\Om).
	\end{align}
	Indeed, if not then 
	there exists $\vartheta>0$ and a subset $\mathcal{S}$ of $\Om$ with $|\mathcal{S}|>0$ such that 
	$$u(x)>\mathcal{C}\|u\|_{L^{k_0  p_s^*}(\Om)}+\vartheta\text{~~for } x\in\mathcal{S},$$
	where $$\mathcal{C}=\left(C_*\right)^{\DD\frac{1}{p_{\mu,s}^*-p}} \left(C^*\right)^{\DD\frac{p-1}{\sqrt {p}(\sqrt{p_{\mu,s}^*}-\sqrt {p})}}.$$ The above implies
	\begin{align*}
	\DD\liminf_{\nu_n\to\infty}\left(\int_{\Om}|u(x)|^{\nu_n}dx\right)^{\frac{1}{\nu_n}}
	&\geq \DD\liminf_{\nu_n\to\infty}\left(\int_{\mathcal{S}}|u(x)|^{\nu_n}dx\right)^{\frac{1}{\nu_n}}\\&\geq\DD\liminf_{\nu_n\to\infty}\left(\mathcal{C}\|u\|_{L^{ k_0 p_s^*}(\Om)}+\vartheta\right)\left(|\mathcal{S}|\right)^{\frac{1}{\nu_n}}\\
	&=\mathcal{C}\|u\|_{L^{k_0  p_s^*}(\Om)}+\vartheta,
	\end{align*}
	a contradiction to \eqref{contr}. Therefore, \eqref{claim} holds. Hence the proof the lemma is complete.
\end{proof}
\begin{proof}[{\bf Proof of Theorem \ref{reg}}]
	Now for proving H\"{o}lder regularity we first claim that 
	\begin{align}\label{12}\left(\int_\Om \frac{F(y,u)}{|x-y|^{\mu}}dy\right) f(x,u)\in L^{\infty}(\Om).\end{align} 
	Indeed, by Lemma \ref{L-infty}, we get $u\in L^{\infty}(\Om)$ and thus, by $(\bf{H}),$ we have $f(\cdot,u(\cdot))$, $F(\cdot,u(\cdot))\in L^\infty(\Om),$ which imply that 
	\begin{align*}\left|\int_\Om \frac{F(y,u)}{|x-y|^{\mu}}dy\right|&\leq \|F(\cdot,u(\cdot))\|_{L^{\infty}(\Om)} 
	\left[\int_{ \Om\cap\{|x-y|<1\}}\frac{dy}{|x-y|^\mu}+\int_{ \Om\cap\{|x-y|\geq1\}}\frac{dy}{|x-y|^\mu}\right]\\
	&\leq \|F(\cdot,u(\cdot))\|_{L^{\infty}(\Om)} \left[\int_{ \Om\cap\{\ol r\leq1\}}\ol {r}^{N-1-\mu}d\ol r+|\Om|\right]\\
	&<\infty,
	\end{align*} 
	and	since $0<\mu<N$, \eqref{12} holds. Now by applying Proposition \ref{regularity}, we
	finally can conclude that there exists some $\al\in[0,s)$, depending upon $s,p,\Om$ such that $u\in C^{0,\al}_d(\ol{\Om}).$ Hence, the proof is complete.
\end{proof}

\begin{proof} [{\bf Proof of Theorem \ref{thm2}}] 
Here we follow the approach as in (\cite{SH}, Theorem 1.1). We consider the two cases separately.\\
		$(a)$\it{Critical Case}: $r=p_{\mu,s}^*$ in $(\bf{H})$:\\
	First, we  show $\rm(i)$ implies $\rm(ii)$. 
	From (i), it follows that
	$\langle J'(w_0),\phi\rangle\geq 0$ for all $\phi\in W_0^{s,p}(\Om)\cap C^{0}_{d}(\overline\Omega)$.
	Since $W_0^{s,p}(\Om)\cap C^{0}_{d}(\overline\Omega)$ is a dense subspace of $W_0^{s,p}(\Om)$, 
	we have $$\langle J'(w_0),\phi\rangle=0\text {\;\; for all } \phi\in W_0^{s,p}(\Om).$$ 
	Therefore, by Theorem \ref{reg} we infer that $w_0\in C^{0}_{d}(\overline\Omega)\cap L^\infty(\Omega)$.
	Here we argue by contradiction. Suppose (ii) does not hold. Then  there exists a sequence, say $\{\tilde w_n\}$ in $W_0^{s,p}(\Om)$
	such that $\tilde w_n\to w_0$ strongly in $W_0^{s,p}(\Om)$ as $n\to\infty$ and $J(\tilde w_n)<J(w_0)$ for all $n\in\N$.
	Next, we introduce a suitable truncation to the nonlinearity $f$ to handle its critical growth (in the sense of Hardy-Littlewood-Sobolev inequality). For each $j\in \N,$ we define $f_j:\Om\times\RR\to\RR$ as
	$f_j(x,t):=f(x,T_j(t))$ where 
	\[T_j(t)=\begin{cases}
		-j &\text{ for }t\leq-j\\
		t &\text{ for }-j\leq t\leq j\\
		j &\text{ for }t\geq j.
	\end{cases}\]
	We define the corresponding truncated energy functional   $J_j:W_0^{s,p}(\Om)\to\RR$ as
	\[J_j(u)=\frac{\|u\|_{s,p}^p}{p}-\int_\Omega\int_\Omega\frac {F_j(x,u) F_j(y,u)}{|x-y|^\mu}dxdy,\] where $F_{j}(x,t)=\int_0^t f_{j}(x,\tau)\,d\tau.$
	One can see that $J_j\in C^1( W_0^{s,p}(\Om))$. Note that, by $(\bf H)$,
	\[|f_{j}(x,t)| \leq \tilde C_j:=K_{0}\,(1+j^{p_{\mu,s}^*-1}), \;\;\;|F_{j}(x, t)|\leq K_{j}(1+|t|)\] are of subcritical growth
	(in the sense of Hardy-Littlewood-Sobolev inequality), where $ \tilde C_j,\, K_{j},$ $j\in\mathbb N,$ are positive real numbers. 
	Now by applying Lebesgue dominated convergence theorem, for all $u\in W_0^{s,p}(\Om),$ we have
	\begin{align}\label{appr}
	\lim_{j\to \infty}F_j(x,u)=\lim_{j\to \infty}\int_0^u f_j(x,t)\,dt =  F(x,u).
	\end{align}
	For fixed $n\in\N$ and $0<\xi_{n}<J(w_0)-J(\tilde w_n)$, using \eqref{appr}, we can find $j_n>C(\|w_0\|_{L^\infty(\Om)})>1$ such that
	{\begin{align}\label{1.0}\Big|\int_\Omega \int_\Omega \frac {F_{j_n}(x,\tilde w_n)F_{j_n}(y,\tilde w_n)}{|x-y|^{\mu}}\,dxdy-
		\int_\Omega\int_\Omega \frac{F(x,\tilde w_n)F(y,\tilde w_n)}{|x-y|^{\mu}}\,dxdy\Big| < \xi_{n}.\end{align}}
	For all $n\in\N,$ let us set 
	\[\sigma_n:=\|\tilde w_n-w_0\|_{L^{\p}(\Om)}, \qquad B_{\sigma_n}:=\big\{ u\in W_0^{s,p}(\Om):\,\| u-w_0\|_{L^{\p}(\Om)}\leq\sigma_n\big\}.\]
	Using the continuous embedding $W_0^{s,p}(\Om)\hookrightarrow  L^{\p}(\Om),$ we have $\sigma_n\to 0$ as $n\to\infty$.
	Now $B_{\sigma_n}$ is a closed convex subset of $W_0^{s,p}(\Om)$  and hence weakly closed  subset of $W_0^{s,p}(\Om)$. Hence by the definition, $J_{j_n}$ is sequentially weakly lower semi-continuous and coercive in $B_{\sigma_n}.$ 
	Thus, for any $n\in\mathbb{N},$ there exists $w_{n}\in  B_{\sigma_n}$ such that
	\begin{align}\label{min}
	J_{j_n}(w_n)=\inf_{u\in B_{\sigma_n}}J_{j_n}(u).
	\end{align}
	In view of \eqref{1.0} and by the choice of $\xi_{n}$ and $j_n$,  we get
	\begin{align}
	\label{jjn}
	J_{j_n}(w_{n})\leq J_{j_n}(\tilde w_n) \leq J(\tilde w_n)+\xi_{n} < J(w_0) = J_{j_n}(w_0).
	\end{align}
	{\bf{Claim:}}  There  exists $m_n\geq 0$ such that  
	\begin{align}\label{lag}
	\fpl w_n+m_n(w_n-w_0)^{\p-1} = \left(\int_{\Om}\frac {F_{j_n}(y,w_n)}{|x-y|^\mu}dy\right){f_{j_n}(x,w_n)}.
	\end{align}
	Since $w_n\in B_{\sigma_n}$, in the process of the proof of our claim we encounter with two possible cases:\\
	{\bf{Case:}}$\|w_n-w_0\|_{L^{\p}(\Om)}<\sigma_n$. Then \eqref{min} yields that $w_n$ is a local minimizer of $J_{j_n}$ in $ W_0^{s,p}(\Om)$ and hence, $J'_{j_n}(w_n)=0$. Thus, \eqref{lag} holds with $m_n=0$.\\\\
	{\bf{Case:}}$\|w_n-w_0\|_{L^{\p}(\Om)}=\sigma_n.$ We define the functional  $\mathcal{I}: W_0^{s,p}(\Om)\to\RR$ as 
	\[\mathcal{I}(u):=\frac{\|u-w_0\|_{L^{\p}(\Om)}^{\p}}{\p}.\]
	One can check that  $\mathcal{I}\in C^1( W_0^{s,p}(\Om),\RR).$ Next, we consider the following $C^1$-manifold in $ W_0^{s,p}(\Om):$ 
	\[\mathscr{M}_n:=\Big\{u\in W_0^{s,p}(\Om):\,\mathcal{I}(u):=\frac{\sigma_n^{\p}}{\p}\Big\}.\] Now \eqref{min} yields that $w_n$ is 
	a global minimizer of $J_{j_n}$ on $\mathscr{M}_n.$ Therefore, by applying Lagrange's multipliers rule, there exists $m_n\in\RR$ such that in $W^{-s,p'}(\Om)$
	\[J'_{j_n}(w_n)+m_n \mathcal{I}'(w_n) = 0,\]
	the PDE form of which is  \eqref{lag}. Furthermore, using \eqref{min} again we can derive
	\[m_n = -\frac{\langle J'_{j_n}(w_n),w_0-w_n\rangle}{\langle \mathcal{I}'(w_n),w_0-w_n\rangle} \ge 0\]
	such that, possibly $m_n\to\infty$ as $n\to\infty.$ Hence, our claim is proved.
	As per the construction,  $w_n\to w_0$ strongly in $L^{\p}(\Omega)$ as $n\to\infty$. Moreover, applying Lemma \ref{L-infty} for \eqref{lag},
	we can have $w_n\in L^\infty(\Omega),$ for all $n\in\mathbb N$.  Next, we will show that, up to a subsequence, $\{w_n\}$ is bounded in $L^\infty(\Omega)$.
	Subtracting \eqref{mainprob} from \eqref{lag}, for all $n\in\N,$ we get
	\begin{align}\label{dif}
	&\fpl w_n-\fpl w_0+m_n(w_n-w_0)^{\p-1}\n&\qquad= \int_\Om\left(\frac{F_{j_n}(y,w_n)}{|x-y|^\mu}dy\right)f_{j_n}(x,w_n)-\int_\Om\left(\frac{F(y,w_0)}{|x-y|^\mu}dy\right)f(x,w_0).
	\end{align}
	We set $v_n:=w_n-w_0\in W_0^{s,p}(\Om)\cap L^\infty(\Omega).$ Then for $\beta:=kp-p+1,~ k\geq 1,$ using  $v_n^{\beta}\in W_0^{s,p}(\Om)$,  as a test function in the weak formulation of \eqref{dif}, we deduce
	\begin{align}\label{pqw}
&\langle\fpl w_n-\fpl w_0,v_n^\beta\rangle+m_n\int_\Omega|v_n|^{\p+\beta-1}\,dx\n&\qquad = 
	\int_\Omega\int_\Omega\frac{F_{j_n}(y,w_n)f_{j_n}(x,w_n)-F(y,w_0)f(x,w_0)}{|x-y|^{\mu}} \;(v_n(x))^\beta\,dxdy.
	\end{align}
	By using Lemma \ref{mon} and the continuous embedding $ W_0^{s,p}(\Om)\hookrightarrow L^{\p}(\Omega),$ from the left-hand side of \eqref{pqw}, we deduce that
	\begin{align}\label{1.1}
	\Big[\int_\Omega |v_n|^{k\p}\,dx\Big]^{\frac{p}{\p}} \leq C\,\Big\|v_n^{{\frac{p+\beta-1}{p}}}\Big\|_{s,p}^p
	\leq C\;\beta^{p-1}\,\langle\fpl w_n-\fpl w_0,v_n^\beta\rangle.
	\end{align} Now we estimate the right-hand side in \eqref{pqw}.
	For that, first observe that by the construction, for $(x,t)\in\Om\times\RR,$ we have $|F_{j_n}(x,t)|\leq|F(x,t)|$ and $|f_{j_n}(x,t)|\leq|f(x,t)|.$ 
	Therefore, using $(\bf H)$ and the fact $w_0\in L^\infty(\Om),$ together with  the inequalities $ (x_1+x_2)^{\gamma}\leq x_1^{\gamma}+x_2^{\gamma},~0<\gamma<1, ~x_1,x_2\geq 0,$ and $ (x_1+x_2)^{\gamma}\leq 2^{\gamma-1} (x_1^{\gamma}+x_2^{\gamma}),~\gamma>1,~ x_1,x_2\geq 0,$ we have 
	\begin{align}\label{difgro}
	&|F_{j_n}(y,w_n)f_{j_n}(x,w_n)-F(y,w_0)f(x,w_0)|\n
	&\leq |F(y,w_n)|\;|f(x,w_n)|+|F(y,w_0)|\;|f(x,w_0)|\nonumber\\
	&\leq C\left[1+\left(1+|(v_n+ w_0)(y)|^{p_{\mu,s}^*}\right)\left(1+|(v_n+ w_0)(x)|^{p_{\mu,s}^*-1}\right)\right]\nonumber\\
	&\leq 2^{2p_{\mu,s}^*}C\Big[1+\Big\{\left(1+|v_n(y)|^{p_{\mu,s}^*}+ |w_0(y)|^{p_{\mu,s}^*}\right)\left(1+|v_n(x)|^{p_{\mu,s}^*-1}+ |w_0(x)|^{p_{\mu,s}^*-1}\right)\Big\}\Big]\nonumber\\
	&\leq \tilde K\left[\left(1+|v_n(y)|^{p_{\mu,s}^*}\right)\left(1+|v_n(x)|^{p_{\mu,s}^*-1}\right)\right]
	\end{align} 
	for some constant $ \tilde{ K}>0$ (independent of $j,n, w_n,v_n$). 
	Let us denote $\widehat{g}(x,t):= 1+|t|^{p_{\mu,s}^*-1}$  and $\widehat{G}(x,t):= 1+|t|^{p_{\mu,s}^*}$.
	Therefore, using \eqref{pqw}-\eqref{difgro} together with $m_n\ge 0$ and Hardy-Littlewood-Sobolev inequality,  for all $n\in\N$ and $k,\beta\ge 1,$ we obtain
	\begin{align}\label{qest}
	&\|v_n\|_{L^{k\p}(\Om)}^{kp}\nonumber\\ &\leq K_2\,\beta^{p-1}\,
	\left[\left\{\left(\int_\Omega|\widehat{g}(x,v_n)~|v_n|^{\beta}|^{\frac{p_s^*}{p_{\mu,s}^*}}\,dx\right)^{\frac{p_{\mu,s}^*}{p_s^*}}\left(\int_\Omega|\widehat{G}(x,v_n)|^{\frac{p_s^*}{p_{\mu,s}^*}}\,dx\right)^{\frac{p_{\mu,s}^*}{p_s^*}}\right\}\right]\nonumber\\
	&\leq K_2\,\beta^{p-1}\,
	\left[\left\{\left(\int_\Omega|v_n|^{\beta\frac{p_s^*}{p_{\mu,s}^*}}\,dx\right)^{\frac{p_{\mu,s}^*}{p_s^*}}+\left(\int_\Omega|v_n|^{(p_{\mu,s}^*+\beta-1)\frac{p_s^*}{p_{\mu,s}^*}}\,dx\right)^{\frac{p_{\mu,s}^*}{p_s^*}}\right\}
	\left\{\left(\|v_n\|_{L^{\frac{p_s^*}{p_{\mu,s}^*}}(\Om)}+\left\|v_n\right\|_{L^{{{p_s^*}}}(\Om)}^{p_{\mu,s}^*}\right)\right\}\right]\nonumber\\
		&\leq K_2\,\beta^{p-1}\,	\left(\|v_n\|_{L^{\frac{p_s^*}{p_{\mu,s}^*}}(\Om)}+\left\|v_n\right\|_{L^{{{p_s^*}}}(\Om)}^{p_{\mu,s}^*}\right)
	\left[\left(\int_\Omega|v_n|^{\beta\frac{p_s^*}{p_{\mu,s}^*}}\,dx\right)^{\frac{p_{\mu,s}^*}{p_s^*}}+\left(\int_\Omega|v_n|^{(p_{\mu,s}^*+\beta-1)\frac{p_s^*}{p_{\mu,s}^*}}\,dx\right)^{\frac{p_{\mu,s}^*}{p_s^*}}
\right]\n
&\leq K_2\,\beta^{p-1}\,	\left(\|v_n\|_{L^{\frac{p_s^*}{p_{\mu,s}^*}}(\Om)}+\left\|v_n\right\|_{L^{{{p_s^*}}}(\Om)}^{p_{\mu,s}^*}\right)
\Bigg[ \left(\int_{\Om\cap\{|v_n|<1\}}|v_n|^{\beta\frac{p_s^*}{p_{\mu,s}^*}}dx\right)^{\frac{p_{\mu,s}^*}{p_s^*}}+\left(\int_{\Om\cap\{|v_n|\geq1\}}|v_n|^{\beta\frac{p_s^*}{p_{\mu,s}^*}}dx\right)^{\frac{p_{\mu,s}^*}{p_s^*}}\n
&\qquad\qquad\qquad\qquad\qquad\qquad\qquad\qquad\qquad\qquad+\left(\int_\Omega|v_n|^{(p_{\mu,s}^*+\beta-1)\frac{p_s^*}{p_{\mu,s}^*}}\,dx\right)^{\frac{p_{\mu,s}^*}{p_s^*}}\Bigg]\n
&\leq K_2\,\beta^{p-1}\,	\left(\|v_n\|_{L^{\frac{p_s^*}{p_{\mu,s}^*}}(\Om)}+\left\|v_n\right\|_{L^{{{p_s^*}}}(\Om)}^{p_{\mu,s}^*}\right)
\Bigg[ \left(\int_{\Om\cap\{|v_n|<1\}}|v_n|^{\beta\frac{p_s^*}{p_{\mu,s}^*}}dx\right)^{\frac{p_{\mu,s}^*}{p_s^*}}\n&\qquad\qquad\qquad\qquad\qquad\qquad+\left(\int_{\Om\cap\{|v_n|\geq1\}}\left(|v_n|^{p_{\mu,s}^*+\beta-1}\right)^{\frac{p_s^*}{p_{\mu,s}^*}}dx\right)^{\frac{p_{\mu,s}^*}{p_s^*}}
+\left(\int_\Omega|v_n|^{(p_{\mu,s}^*+\beta-1)\frac{p_s^*}{p_{\mu,s}^*}}\,dx\right)^{\frac{p_{\mu,s}^*}{p_s^*}}\Bigg]\n
&\leq K_2\,\beta^{p-1}\,	\left(\|v_n\|_{L^{\frac{p_s^*}{p_{\mu,s}^*}}(\Om)}+\left\|v_n\right\|_{L^{{{p_s^*}}}(\Om)}^{p_{\mu,s}^*}\right)
\left[ \left(\int_{\Om\cap\{|v_n|<1\}}|v_n|^{\beta\frac{p_s^*}{p_{\mu,s}^*}}dx\right)^{\frac{p_{\mu,s}^*}{p_s^*}}
+2\left(\int_\Omega|v_n|^{(p_{\mu,s}^*+\beta-1)\frac{p_s^*}{p_{\mu,s}^*}}\,dx\right)^{\frac{p_{\mu,s}^*}{p_s^*}}\right]\n
&\leq K_2\,\beta^{p-1}\,	\left(\|v_n\|_{L^{\frac{p_s^*}{p_{\mu,s}^*}}(\Om)}+\left\|v_n\right\|_{L^{{{p_s^*}}}(\Om)}^{p_{\mu,s}^*}\right)
\Bigg[|\Om|^{\frac{p_{\mu,s}^*-1}{p_{\mu,s}^*+\beta-1}.\frac{p_{\mu,s}^*}{p_s^*}}\left\{\left(\int_{\Om\cap\{|v_n|<1\}}|v_n|^{\frac{p_s^*}{p_{\mu,s}^*}{(p_{\mu,s}^*+\beta-1)}}dx\right)^{\frac{\beta}{p_{\mu,s}^*+\beta-1}}\right\}^{\frac{p_{\mu,s}^*}{p_s^*}}\n&\qquad\qquad\qquad\qquad\qquad\qquad\qquad\qquad\qquad\qquad+2\left(\int_\Omega|v_n|^{(p_{\mu,s}^*+\beta-1)\frac{p_s^*}{p_{\mu,s}^*}}\,dx\right)^{\frac{p_{\mu,s}^*}{p_s^*}}\Bigg]\n
	&\leq K_2\,\beta^{p-1}\,	\left(\|v_n\|_{L^{\frac{p_s^*}{p_{\mu,s}^*}}(\Om)}+\left\|v_n\right\|_{L^{{{p_s^*}}}(\Om)}^{p_{\mu,s}^*}\right)
	\Bigg[|\Om|^{\frac{p_{\mu,s}^*-1}{p_{\mu,s}^*+\beta-1}.\frac{p_{\mu,s}^*}{p_s^*}}\left\{\left(\int_{\Om\cap\{|v_n|<1\}}|v_n|^{\frac{p_s^*}{p_{\mu,s}^*}{(p_{\mu,s}^*+\beta-1)}}dx\right)^{\frac{\beta}{p_{\mu,s}^*+\beta-1}}\right\}^{\frac{p_{\mu,s}^*}{p_s^*}}\n&\qquad\qquad\qquad\qquad\qquad\qquad\qquad\qquad\qquad\qquad+
	2\left\|v_n\right\|_{L^{{{p_s^*}}}(\Om)}^{p_{\mu,s}^*-p}\left\|v_n\right\|_{L^{{{kp_s^*}}}(\Om)}^{kp}\Bigg]
\end{align}
	with some constant $K_2>0,$ independent of $k,\beta,j,n, v_n$.
	Now two cases arise for the first integration expression in the right-hand side of \eqref{qest}.
\begin{align*}
				&\text{Either  } \left(\int_{\Om\cap\{|v_n|<1\}}|v_n|^{\frac{p_s^*}{p_{\mu,s}^*}{(p_{\mu,s}^*+\beta-1)}}dx\right)^{\frac{\beta}{p_{\mu,s}^*+\beta-1}}\leq 1 \text{  or  }>1, \text{\;\; that is,}\n
				&\text{ either  }\left(\int_{\Om\cap\{|v_n|<1\}}|v_n|^{\frac{p_s^*}{p_{\mu,s}^*}{(p_{\mu,s}^*+\beta-1)}}dx\right)^{\frac{\beta}{p_{\mu,s}^*+\beta-1}}\leq 1\\ &\qquad\qquad\qquad\qquad \text{  or  }\\ &\left(\int_{\Om\cap\{|v_n|<1\}}|v_n|^{\frac{p_s^*}{p_{\mu,s}^*}{(p_{\mu,s}^*+\beta-1)}}dx\right)^{\frac{\beta}{p_{\mu,s}^*+\beta-1}}<\int_{\Om\cap\{|v_n|<1\}}|v_n|^{\frac{p_s^*}{p_{\mu,s}^*}{(p_{\mu,s}^*+\beta-1)}}dx.
		\end{align*}
		The above implies 
	\begin{align}\label{n12}
	\left\{	\left(\int_{\Om\cap\{|v_n|<1\}}|v_n|^{\frac{p_s^*}{p_{\mu,s}^*}(p_{\mu,s}^*+\beta-1)}dx\right)^{\frac{\beta}{p_{\mu,s}^*+\beta-1}}\right\}^{\frac{p_s^*}{p_{\mu,s}^*}}&\leq	1+\left\{\int_{\Om\cap\{|v_n|<1\}}\left(|v_n|^{p_{\mu,s}^*+\beta-1}\right)^{\frac{p_s^*}{p_{\mu,s}^*}}dx\right\}^{\frac{p_s^*}{p_{\mu,s}^*}}\n&\leq1+\left(\int_{\Om\cap\{|v_n|<1\}}\left(|v_n|^{p+\beta-1}\right)^{\frac{p_s^*}{p_{\mu,s}^*}}dx\right)^{\frac{p_{\mu,s}^*}{p_s^*}}\n&\leq1+\|v_n\|_{L^{\frac{kpp_s^*}{p_{\mu,s}^*}}(\Om)}^{kp}.
	\end{align} 
	Now using the fact that $v_n\to 0$ strongly in $L^{\p}(\Omega)$ as $n\to\infty$, we eventually have $v_n\to 0$ strongly in $L^{\frac{\p}{p_{\mu,s}^*}}(\Omega)$ as $n\to\infty$. So, for sufficiently large $n\in\mathbb N$, we can have $\left(\|v_n\|_{L^{\frac{p_s^*}{p_{\mu,s}^*}}(\Om)}+\left\|v_n\right\|_{L^{{{p_s^*}}}(\Om)}^{p_{\mu,s}^*}\right):=\e_n\to 0$. By plugging \eqref{n12}  into \eqref{qest}, we obtain
	\begin{align*}
	\|v_n\|_{L^{k\p}(\Om)}^{kp}&\leq K_3\,\beta^{p-1}\,	\e_n
	\left[\left(\int_\Omega|v_n|^{\beta\frac{p_s^*}{p_{\mu,s}^*}}\,dx\right)^{\frac{p_{\mu,s}^*}{p_s^*}}+\left(\int_\Omega|v_n|^{(p_{\mu,s}^*+\beta-1)\frac{p_s^*}{p_{\mu,s}^*}}\,dx\right)^{\frac{p_{\mu,s}^*}{p_s^*}}
	\right]\nonumber\\
	&\leq K_3\,\beta^{p-1}\,	\e_n
	\left[1+\left\|v_n\right\|_{L^{{{\frac{kpp_s^*}{p_{\mu,s}^*}}}}(\Om)}^{kp}+\left\|v_n\right\|_{L^{{{p_s^*}}}(\Om)}^{p_{\mu,s}^*-p}\left\|v_n\right\|_{L^{{{kp_s^*}}}(\Om)}^{kp}
	\right],
	\end{align*} where the	constant ${K_3}>0$ is independent of $j,n,k,\beta,v_n$.
Now  choosing $n\in \mathbb N$ sufficiently large  in the above inequality such that $K_3\left\|v_n\right\|_{L^{{{p_s^*}}}(\Om)}^{p_{\mu,s}^*-p}\beta^{p-1}\e_n<\frac 12$ , we get 
\begin{align*}
\|v_n\|_{L^{k\p}(\Om)}^{kp}&\leq 2 K_3\,\beta^{p-1}\e_n\left[1+\left\|v_n\right\|_{L^{{{\frac{kpp_s^*}{p_{\mu,s}^*}}}}(\Om)}^{kp}\right]\n&\leq 2 K_3\,\beta^{p-1}\e_n\left[1+\left\|v_n\right\|_{L^{{{{kp_s^*}}}}(\Om)}^{kp}\right].
\end{align*} 
Again, by taking $\e_n\leq\frac{1}{4 K_3\,\beta^{p-1}}$ for sufficiently large $n\in\mathbb N$ in the above inequality, we get \begin{align*}
\|v_n\|_{L^{k\p}(\Om)}^{kp}&\leq 4 K_3\,\beta^{p-1}\e_n<1.
\end{align*} 
This implies that $\|v_n\|_{L^{k\p}(\Om)}\leq1^{1/kp}$.
Now, letting $k\to \infty$, we infer that
$\|v_n\|_{L^{\infty}(\Om)}\leq  1$ { for sufficiently large } $n\in\mathbb N.$ 
Thus $\{v_n\}$ is  bounded in $L^\infty(\Omega)$ and  hence the boundedness of $\{w_n\}$ in $ L^{\infty}(\Om)$ follows.
	\vskip2pt
	\noindent
	Now for sufficiently large $n$ $\in\N$, \eqref{lag} can be rewritten as
	\begin{align}\label{lag1}
	\fpl w_n = \left(\int_\Om\frac{F(y,w_n)}{|x-y|^\mu}dy\right)f(x,w_n)-m_n(w_n-w_0)^{\p-1} \ \text{in $W^{-s,p'}(\Om)$.}
	\end{align}
	Since   $\{w_n\}$ is bounded in $L^\infty(\Omega)$ and 
	hence, by $(\bf H),$ it follows that the sequence $\left\{\left(\int_\Om\frac{F(y,w_n)}{|x-y|^\mu}dy\right)f(\cdot,w_n)\right\}$ is also uniformly bounded.
	Therefore, using this fact and  again using $v_n^\beta$ (where $v_n=w_n-w_0$, $\beta=kp+p-1,\,k\ge 1$) as a test function in the weak formulation of \eqref{dif}  and applying Lemma \ref{mon}, 
	for all sufficiently large $n\in\mathbb N$, we achieve
	\begin{align*}
	m_n\,\int_\Omega|v_n|^{\p+\beta-1}\,dx
	&\leq K_4\,\int_\Omega |v_n|^{\beta}\,dx \\
	&\leq K_4\,\Big[\int_\Omega|v_n|^{\p+\beta-1}\,dx\Big]^{\frac{\beta}{\p+\beta-1}}|\Omega|^{\frac{\p-1}{\p+\beta-1}},
	\end{align*}
	where $K_4>0$ is a constant that is independent of $j,n$, $k,\beta,v_n$. The above  implies
	\[m_n\,\|v_n\|_{L^{\p+\beta-1}(\Om)}^{\p-1} \leq K_5\,|\Omega|^{\frac{\p-1}{\p+\beta-1}},\]
	with some constant $K_5>0$ (independent of $j,n,k,\beta,v_n$ ).
	Letting $k\to\infty$, we have $\beta\to\infty$ and hence, from the last relation, we deduce 
	that \[m_n\,\|v_n\|_{L^\infty(\Om)}^{\p-1}\leq K_5,\]
	that is, $\{m_n\,(w_n-w_0)^{\p-1}\}$ is a bounded sequence in $L^\infty(\Omega)$. Hence, combining this fact along with 
	\eqref{lag1} and Theorem \ref{reg}, we see that $\{w_n\}$ is bounded in {$C^{0,\alpha}_d(\overline\Omega)$}. By the compact embedding 
	${C^{0,\alpha}_d(\overline\Omega)}\hookrightarrow\hookrightarrow C_{d}^{0}(\bar \Omega)$, passing to a subsequence, still denoted by
	$\{w_n\},$ we have $w_n\to w_0$ strongly in $C_{d}^{0}(\overline\Omega)$ as $n\to\infty$. So, for all $n\in\mathbb N$ large enough,
	we obtain that $\|w_n-w_0\|_{C^{0}_{d}(\ol\Om)}\leq\varrho$.
	\par
	On the other hand, since $\{w_n\}$ is bounded in $L^{\infty}(\Omega)$,  we get $J_{j_n}(w_{n})=J(w_{n})$ for sufficiently large $n\in\mathbb N.$ 
	Hence, from  \eqref{jjn}, it  follows that $J(w_n)<J(w_0)$. Thus, we reach at a contradiction to (i) and hence, (ii) is proved.\\
	\par Next, we show that $\rm(ii)$ implies $\rm(i)$. By (ii), we have $\langle J'(v_0),v\rangle=0,$ for all $v\in W_0^{s,p}(\Om)$.
	Therefore, Lemma \ref{L-infty} and Proposition \ref{regularity} imply that $w_0\in C_d^{0}(\ol{\Om}).$ Supposing the contrary, let there exist a sequence
	$\{w_n\}$ in $W_0^{s,p}(\Om)\cap C_d^{0}(\overline{\Om})$ such that $w_n\to w_0$ in $C_d^{0}(\overline{\Om})$ as $n\to\infty$
	and $J(w_n)<J(w_0),$ for all $n\in\mathbb N.$ Thus, we have $w_n\to w_0$ strongly in $L^{\infty}(\Om)$ as $n\to\infty$.
	Hence, by the continuity  and (${\bf H}$), the sequence $\{F(\cdot,w_n)\}$ is bounded in  $L^\infty(\Om)$ and
	\begin{align}\label{ff}
	F(\cdot,w_n)\to F(\cdot, w_0)  \text {\;\;strongly in\;\;} L^\infty(\Om)  \text{\;\;as\;} n\to\infty.
	\end{align}
	Observe that  
	\begin{align}\label{lim}
	I^{(n)}_0:&=\int_{\Om}\int_{\Om}\frac{F(x,w_n)F(y,w_n)}{|x-y|^{\mu}}dxdy
	- \int_{\Om}\int_{\Om}\frac{F(x,w_0)F(y,w_0)}{|x-y|^{\mu}}dxdy\nonumber\\
	&=I^{(n)}_1+I^{(n)}_2,
	\end{align} 
	where \begin{align*}
	I^{(n)}_1&:=\int_{\Om}\int_{\Om}\frac{F(x,w_n)\left[F(y,w_n)-F(y,w_0)\right]}{|x-y|^{\mu}}dxdy,\nonumber\\
	I^{(n)}_2&:=\int_{\Om}\int_{\Om}\frac{\left[F(x,w_n)-F(x,w_0)\right]F(y,w_0)}{|x-y|^{\mu}}dxdy.
	\end{align*}
	By Hardy-Littlewood-Sobolev inequality and \eqref{ff}, we get
	\begin{align}\label{i1}
	|I^{(n)}_1|&\leq \|F(\cdot,w_n)\|_{L^{\frac{2N}{2N-\mu}}(\Om)}\|F(\cdot,w_n)-F(\cdot,w_0)\|_{L^{\frac{2N}{2N-\mu}}(\Om)}\nonumber\\
	&\quad\to0 \text{\quad\; as\;} n\to\infty.
	\end{align}Arguing similarly,  we obtain $|I^{(n)}_2|\to 0$ as $n\to\infty$,  which together with \eqref{i1} and \eqref{lim} implies that $I_0\to 0,$ that is,
	\begin{align}\label{t1}
	\int_{\Om}\int_{\Om}\frac{F(x,w_n)F(y,w_n)}{|x-y|^{\mu}}dxdy
	\to\int_{\Om}\int_{\Om}\frac{F(x,w_0)F(y,w_0)}{|x-y|^{\mu}}dxdy \text{\;\;as\;} n\to\infty.
	\end{align}
	Using \eqref{t1},
	we have
	\begin{align}\label{t2}
	\limsup_{n\to\infty}\frac{\|w_n\|_{s,p}^p}{p}&=\limsup_{n\to\infty}\left[J(w_n)+\frac12\int_{\Om}\int_{\Om}\frac{F(x,w_n)F(y,w_n)}{|x-y|^{\mu}}dxdy\right]\n
	&\leq J(w_0)+\frac{1}{2}\int_{\Om}\int_{\Om}\frac{F(x,w_0)F(y,w_0)}{|x-y|^{\mu}}dxdy\n
	&=\frac{\|w_0\|_{s,p}^p}{p},
	\end{align}
	that is, $\{w_n\}$ is bonded in $W_0^{s,p}(\Om).$ Hence, passing to a subsequence, still denoted by $\{w_n\},$ we have $w_n\rightharpoonup w_0$ 
	weakly in $W_0^{s,p}(\Om)$ as $n\to\infty.$ Therefore, using the lower semicontinuity of norm,
	\begin{align}\label{t3}
	\liminf_{n\to\infty}\|w_n\|_{s,p}\geq\|w_0\|_{s,p}.
	\end{align}
	Since $W_0^{s,p}(\Om) $ is uniformly convex, we get $\|w_n\|_{s,p}\to\|w_0\|_{s,p}$ and consequently 
	Brezis-Lieb lemma gives that $\|w_n-w_0\|_{s,p}\to 0$ as $n\to\infty.$ So we have, for  large $n\in\mathbb N$, $\|w_n-w_0\|_{s,p}\leq\delta$ with $J(w_n)<J(w_0),$ which is a contradiction. Therefore, $\rm(i)$ holds.\\
	{$(b)$\it{ Subcritical Case}}: $r<p_{\mu,s}^*$ in $(\bf{H})$:\\
	In this case, the proof follows using the similar arguments as in the {\it {Critical Case}}, discussed above, by taking $f_j(x,t)=f(x,t)$ for each $j\in\mathbb N$ and  using the compact embedding $W_0^{s,p}(\Om)\hookrightarrow\hookrightarrow L^q(\Om),\;\;1<q<p_s^*.$ Hence,  the proof of the theorem is complete.
\end{proof}
\noindent
\begin{proof}[{\bf Proof of Theorem \ref{thmch4}}]
	First, we define the following truncated function  $\widehat{f}(x,t):\Omega\times\RR\to\RR$ as
	\[\widehat{f}(x,t):=
	\begin{cases}
	f(x,\underline{v}(x)) \ & \text{ if } t \le \underline{v}(x) \\
	f(x,t) \ & \text{ if } \underline{v}(x) < t < \overline{v}(x) \\
	f(x, \overline{v}(x)) & \text{ if } t \ge \overline{v}(x).
	\end{cases}\]
	Clearly by $(\bf{H}),$ we have that  $\widehat{f}$ is  continuous such that for a.e. $x\in\Omega$ and all $t\in\RR$
	\begin{equation}\label{ft}
	\;\;\;	\left.\begin{array}{rl}	
	|\widehat{f}(x,t)| &\le C_1 ( 1 + |\underline{v}|^{r-1} + |\overline{v}|^{r-1}),\\
	|\widehat{F}(x,t)|&=\left|\int_{0}^{t}\widehat{f}(x,\tau)d\tau\right|
	\le C_2 ( 1 + |\underline{v}|^{r-1} + |\overline{v}|^{r-1})|t|,
	\end{array}
	\right\}
	\end{equation}
	for some constants $C_1,C_2>0.$
	We define the operator  $T: W_0^{s,p}(\Om) \rightarrow W^{-s,p'}(\Om)$  as 
	\[\langle T(u),v\rangle = - \int_{\Omega}\int_{\Omega} \frac{\widehat F(y,u)\widehat{f}(x,u)v(x)}{|x-y|^\mu}\,dxdy,\text{\; for all $u,v\in W_0^{s,p}(\Om)$}.\] 
	In view of Hardy-Littlewood-Sobolev inequality and \eqref{ft}, $T$ is
	well posed. 
	We will show that $T$ is strongly continuous (see \cite[Definition 2.95 $(iv)$]{CLM}). Indeed, let $\{u_n\}$ be a sequence in $W_0^{s,p}(\Om)$ such that
	\ $u_n \rightharpoonup u_*$ weakly  in $W_0^{s,p}(\Om)$ as $n\to \infty$. Now using compact embedding $W_0^{s,p}(\Om)\hookrightarrow \hookrightarrow L^q(\Omega), \;\;1<q<p_s^*,$
	passing to a subsequence, still denoted by $\{u_n\},$ we have $u_n \rightarrow u_*$ strongly in $L^q(\Omega)$, $u_n(x) \rightarrow u_*(x)$ as $n\to\infty$ and hence,
	there exists some $\tilde h\in L^{\frac{2N}{2N-\mu}}(\Om)$ with $|u_n(x)|\leq \tilde h(x)$ for a.e. $x\in \Om.$ Moreover, for a.e. $x\in\Om,$ we have $\widehat{f}(x,u_n(x))\to\widehat{f}(x,u_*(x))$ and $\widehat{F}(x,u_n(x))\to\widehat{F}(x,u_*(x))$ as $n\to\infty$ and $\{F(\cdot,u_n)\}$ is bounded in $L^{\frac{2N}{2N-\mu}}(\Om)$ and thus, $F(\cdot,u_n)\rightharpoonup F(\cdot,u_*)$ weakly in $L^{\frac{2N}{2N-\mu}}(\Om)$  as $n\to\infty.$ For $v\in W_0^{s,p}(\Om),$ consider the linear continuous  map $\Sigma: L^{\frac{2N}{2N-\mu}}(\Om)\to\RR,$ defined as
	\[\Sigma(w)=\int_\Om\int_\Om\frac{w(y)\widehat f(x,u_*)v(x)}{|x-y|^{\mu}}dxdy.\] Then by letting $n\to\infty,$ we get
	\begin{align}\label{int2}
	\int_\Om\int_\Om\frac{\widehat F(y,u_n)\widehat f(x,u_*)v(x)}{|x-y|^{\mu}}dxdy\to \int_\Om\int_\Om\frac{\widehat F(y,u_*)\widehat f(x,u_*)v(x)}{|x-y|^{\mu}}dxdy.
	\end{align} Therefore, using Hardy-Littlewood-Sobolev inequality, \eqref{ft}, \eqref{int2} and  Lebesgue dominated convergence theorem, we obtain
	\begin{align}
	&|\langle T(u_n)-T(u_*),v\rangle|\nonumber\\&\leq \left|\int_{\Om}\int_{\Om}\frac{\widehat F(y,u_n)\left[(\widehat f(x,u_n)-\widehat f(x,u_*))v\right]}{|x-y|^{\mu}}dxdy\right|
	+\left|\int_{\Om}\int_{\Om}\frac{\left[\widehat F(y,u_n)-\widehat F(y,u_*)\right]\widehat f(x,u_*)v}{|x-y|^{\mu}}dxdy\right|\nonumber\\
	&\leq \|\widehat F(\cdot,u_n)\|_{L^{\frac{2N}{2N-\mu}}(\Om)}\|(\widehat f(\cdot,u_n)-\widehat f(\cdot,u_*))v\|_{L^{\frac{2N}{2N-\mu}}(\Om)}
	+\left|\int_{\Om}\int_{\Om}\frac{\left[\widehat F(y,u_n)-\widehat F(y,u_*)\right]\widehat f(x,u_*)v}{|x-y|^{\mu}}dxdy\right|\nonumber\\
	&\qquad\to 0 \text{\;\; as } n\to \infty.
	\end{align}
	That is,
	$T(u_n) \rightarrow T(u_*)$ in $W^{-s,p'}(\Om)$ as $n\to\infty$ and thus,  $T$ is strongly continuous. Hence, \cite[Lemma 2.98  $(ii)$]{CLM} yields that $T$ is pseudomonotone.\\
	Next, using Lemma \ref{mon} and arguing as in \cite[Lemma 3.2]{ss} we get that $\fpl:W_0^{s,p}(\Om)\to W^{-s,p'}(\Om)$ is pseudomonotone.
	Therefore, $\fpl+T$ is a pseudomonotone operator.
	\vskip2pt
	\noindent
	On the other hand, using Hardy-Littlewood-Sobolev inequality, H\"older inequality, \eqref{ft} and the continuous embedding $W_0^{s,p}(\Om)\hookrightarrow L^q(\Omega), \;\; 1<q<p_s^*$, we deduce
	\begin{align*}
	\|T(u)\|_{-s,p'} &= \sup_{\|v\|_{s,p} \le 1}\left| \int_{\Om}\int_{\Omega} \frac{\widehat{F}(y,u)\widehat{f}(x,u)v(x)} {|x-y|^\mu}\,dxdy \right|\\
	&\leq C_3\|\widehat{f}(\cdot,u)v\|_{L^{\frac{2N}{2N-\mu}}(\Om)} \|\widehat{F}(\cdot,u)\|_{L^{\frac{2N}{2N-\mu}}(\Om)} \\
	&\leq C_4 \left(\|v\|_{L^{\frac{2N}{2N-\mu}}(\Om)} + \|\underline{v}\|_{L^{r\frac{2N}{2N-\mu}}(\Om)}^{r-1} \|v\|_{L^{r\frac{2N}{2N-\mu}}(\Om)} + \|\overline{v}\|_{L^{r\frac{2N}{2N-\mu}}(\Om)}^{r-1}\|v\|_{L^{r\frac{2N}{2N-\mu}}(\Om)}\right)^2\\
	&\leq C_5 \|v\|_{s,p}^2\left(1 + \|\underline{v}\|_{L^{r\frac{2N}{2N-\mu}}(\Om)}^{r-1} + \|\overline{v}\|_{L^{r\frac{2N}{2N-\mu}}(\Om)}^{r-1}\right)^2\\
	&\leq C_5 \left(1 + \|\underline{v}\|_{L^{r\frac{2N}{2N-\mu}}(\Om)}^{r-1} + \|\overline{v}\|_{L^{r\frac{2N}{2N-\mu}}(\Om)}^{r-1}\right)^2,
	\end{align*} where $C_3,C_4,C_5$ are some positive constants which do not depend on $u,v$. This implies that $T$ is bounded. Again arguing as in \cite[Lemma 3.2]{ss}, it follows that $\fpl+T$ is bounded.
	\vskip2pt
	\noindent
	Finally we show that $\fpl+T$ is coercive. Indeed, again using Hardy-Littlewood-Sobolev inequality, H\"older inequality, \eqref{ft} and 
	the continuous embedding $W_0^{s,p}(\Om)\hookrightarrow L^q(\Omega), \;\; 1<q<p_s^*$, for all $u \in W_0^{s,p}(\Om) \setminus \{0\},$ we have
	\begin{align*}
	\frac{\langle \fpl(u) + T(u),u\rangle} {\|u\|_{s,p}} &\;\;=  \|u\|_{s,p}^{p-1} - \frac{1}{\|u\|_{s,p}} \int_{\Om}\int_{\Omega} \frac{\widehat{F}(y,u)\widehat{f}(x,u)u(x)} {|x-y|^\mu}\,dxdy \\
	&\;\; \ge \|u\|_{s,p}^{p-1} - \frac{C_6}{\|u\|_{s,p}} \Big(\|u\|_{L^{\frac{2N}{2N-\mu}}(\Om)} + \|\underline{v}\|_{L^{r\frac{2N}{2N-\mu}}(\Om)}^{\frac{\mu}{2N-\mu}} \|u\|_{L^{r\frac{2N}{2N-\mu}}(\Om)}\nonumber\\&\qquad\qquad\qquad\qquad\qquad\qquad+ \|\overline{v}\|_{{L^{r\frac{2N}{2N-\mu}}(\Om)}}^{\frac{\mu}{2N-\mu}} \|u\|_{L^{r\frac{2N}{2N-\mu}}(\Om)} \Big)^2\\
	&\;\; \ge \|u\|_{s,p}^{p-1} - C_7\|u\|_{s,p}\to\infty \text{\qquad as\; } \|u\|_{s,p}\to \infty,
	\end{align*}
	where  $C_6, C_7>0$ are some constants, independent of $u$.  Applying \cite[Theorem 2.99]{CLM}, we get that there exists a solution, say $v_0\in W_{0}^{s,p}(\Om)$ 
	to the following equation:
	\begin{align}\label{eq}
	\fpl u+T(u)=0 \text{ in }W^{-s,p'}(\Om),
	\end{align} that is, $$\widehat{J}(v_0)=\min_{u\in W_0^{s,p}(\Om)}\widehat{J}(u),$$ where  $\widehat{J}\in C^1(W_0^{s,p}(\Om),\RR)$ is the
	energy functional associated to  \eqref{eq} and is defined as $$\widehat{J}(u)=\frac{\|u\|_{s,p}^p}{p}-\int_{\Om}\int_{\Om}\frac{\widehat{F}(x,u)\widehat{F}(y,u)}{|x-y|^\mu}dxdy.$$ 
	Now we claim that
	\begin{align}\label{int1}
	\underline{v} \le v_0 \le \overline{v} \text{\;\;\;\;in\; $\Omega$}.
	\end{align}
	Observe that, \eqref{int1} holds in $\RR^N\setminus\Omega$. Using $(v_0-\overline{v})^+ \in W_0^{s,p}(\Om)$ as a test function in the weak formulation of \eqref{eq} and using the fact that $\overline{v}$ is a supersolution of \eqref{mainprob}, we deduce
	\begin{align*}
	\langle \fpl v_0, (v_0-\overline{v})^+\rangle &= \int_{\Omega}\int_{\Omega}\frac{\widehat{F}(y,v_0)}{|x-y|^\mu} \widehat{f}(x,v_0) (v_0-\overline{v})^+(x) \,dxdy \\
	&= \int_{\Omega} \int_{\Omega}\frac{{F}(y,\overline{v})}{|x-y|^\mu} f(x, \overline{v}) (v_0-\overline{v})^+(x) \,dx dy\\
	&\le \langle \fpl \overline{v}, (v_0-\overline{v})^+\rangle.
	\end{align*}
	Therefore, we get 
	\begin{align}\langle\label{ss} \fpl v_0 - \fpl\overline{v}, (v_0-\overline{v})^+\rangle \le 0.\end{align}
	From \cite[Lemma A.2]{second-eigenvalue} and \cite[Lemma 2.3] {SH} (with $g(t)=t^+$),
	we recall the following standard inequalities: for all $a,b \in \RR,$
	\[|a^+-b^+|^p \le (a-b)^{p-1} (a^+ - b^+), \quad (a-b)^{p-1} \le C_p (a^{p-1}-b^{p-1}).\]
	Using the last two inequalities along with \eqref{ss}, we obtain
	\begin{align*}
	&\|(v_0-\overline{v})^+\|_{s,p}^p = \int_{\RR^N}\int_{\RR^N} \frac{|(v_0(x)-\overline{v}(x))^+ - (v_0(y)-\overline{v}(y))^+|^p}{|x-y|^{N+sp}}\,dxdy\\
	& \leq \int_{\RR^N}\int_{\RR^N}\frac{ [(v_0-\overline{v})(x) - (v_0-\overline{v})(y)]^{p-1}[(v_0-\overline{v})^+(x) - (v_0-\overline{v})^+(y)]}{|x-y|^{N+sp}}\,dxdy\\
	& \leq C_p\bigg[ \int_{\RR^N}\int_{\RR^N} \frac{(v_0(x) - v_0(y))^{p-1}  [(v_0-\overline{v})^+(x) - (v_0-\overline{v})^+(y)]}{|x-y|^{N+sp}}dxdy\\&\quad\qquad-  \int_{\RR^N}\int_{\RR^N} \frac{(\overline{v}(x)-\overline{v}(y))^{p-1} [(v_0-\overline{v})^+(x) - (v_0-\overline{v})^+(y)]}{|x-y|^{N+sp}}dxdy\bigg]\\
	&=C_p \langle \fpl v_0 - \fpl \overline{v}, (v_0-\overline{v})^+\rangle\le 0,
	\end{align*}
	and thus, we have $(v_0-\overline{v})^+ = 0$ which infers that $v_0 \le \overline{v}$ in $\Omega$. Similarly, testing \eqref{eq} with $(v_0-\underline{v})^- \in W_0^{s,p}(\Om),$
	we can prove that $v_0 \ge \underline{v}$ and hence, \eqref{int1} holds true. \\
	Next, we show that $v_0$ is a local minimizer of $J$ in $W_0^{s,p}(\Om)$.
	By exploiting the monotonicity  and the definition of $\widehat{f}$ along with \eqref{int1},
	we obtain, in the weak sense, \begin{align*}
	(-\Delta)_p^s(\overline{v}-v_0)& \geq  \left( \int_{\Om }\frac{F(y,\overline{v})}{|x-y|^{\mu}}~dy\right) f(x,\overline{v})- \left( \int_{\Om }\frac{\widehat{F}(y,v_0)}{|x-y|^{\mu}}~dy\right) \widehat{f}(x,v_0)\\
	& \geq 0 
	\end{align*} in $\Om$ and by definition, $\overline{v}-v_0\geq 0$ in $\RR^N \setminus \Om.$
	In view of the fact that $\overline{v}$ is not a solution to \eqref{mainprob}, we have  $v_0 \not = \overline{v}.$ 
	Therefore, by the strong maximum principle for fractional $p$-Laplacian (Lemma \ref{max}),   it follows that $\overline{v}-v_0>0$ in $\Om$ and similarly  $v_0-\underline{v}>0$ in $\Om.$ Thus, it follows that $v_0 \in W_0^{s,p}(\Om)$ is a weak solution to  \eqref{mainprob}.
	Now, again using Lemma \ref{max} and Hopf's lemma for fractional $p$-Laplacian (see \cite[Theorem 1.5]{DQ}), 
	we can have   $\overline{v}-v_0 \geq Rd^s$ in $\Om,$ for some $R>0$, where $d$ is the distance function, defined in \eqref{dist}. Likewise, $v_0 -\underline{v}\geq Rd^s$ in $\Om,$ for some $R>0$. 
	Also, from Lemma \ref{L-infty} and Proposition \ref{regularity}, we get $ v_0 \in C^{0,\al}_d(\overline{\Om})$. Let us denote $$\bar{B}^d_{R/2}(v_0):=\{u\in W_0^{s,p}(\Om): \|u-v_0\|_{C^{0}_d(\ol\Om)}\;\leq R/2\}.$$ For each $w \in \bar{B}^d_{R/2}(v_0)$, we have  
	\begin{align*}
	\frac{\overline{v}-w}{d^s} = \frac{\overline{v} - v_0}{d^s}+ \frac{v_0 -w}{d^s} \geq R- \frac{R}{2}=\frac R2 \text{\;\;\; in\;} \overline{\Om}. 
	\end{align*}
	The last relation implies that $\overline{v}-w>0$ in $\Om$. On a similar note, we have $w -\underline{v}>0 $  in $\Om$. Therefore, in $W_0^{s,p}(\Om)\cap\bar{B}^d_{R/2}(v_0),$  $\widehat{J}$ agrees with $J$ and thus,  $v_0$ emerges as a local minimizer of $J$ in $W_0^{s,p}(\Om)\cap \bar{B}^d_{R/2}(v_0).$ Finally, from Theorem \ref{thm2}, we infers that $v_0$ is a local minimizer of $J$ in $W_0^{s,p}(\Om)$ as well  and hence, the theorem is proved. 
\end{proof}
{
	\noi\begin{proof}[{\bf Proof of Theorem \ref{pert} and \ref{thmch4.1}}]
		The proofs of Theorem \ref{pert} and Theorem \ref{thmch4.1} follow using the similar arguments as
		in the proofs of Theorems \ref{reg}-\ref{thmch4},  by incorporating appropriate modifications in the calculations using 
		Sobolev embedding $W^{s,p}_0(\Omega) \hookrightarrow L^{q}(\Om)$, $1<q\leq p_s^*,$ for the additional perturbation term $g(x,u)$ and defining the required   truncation for $g$ in a similar fashion as defined for $f$ in the previous proofs. \end{proof}}
	As discussed in \cite{divya}, Theorem \ref{thmch4.1} finds application in proving the multiplicity of the solutions of the problem \eqref{mainprob1} with concave type perturbation $g(x,u)$. 
	In particular consider the following problem:
	\begin{equation}\label{ch33}
	\left.\begin{array}{rllll}
	(-\Delta)_p^s u
	&=\la \left( |u|^{q-2}u + \left(\DD \int_{\Om}\frac{F(y,u)}{|x-y|^{\mu}}dy\right) f(x,u)\right) ,\; u>0 \; \text{in}\;
	\Om,\\
	u&=0 \; \text{ in } \RR^N \setminus \Om,
	\end{array}
	\right\}
	\end{equation}
	where $\la>0,\; 1<q<p$ and $f$  is a non decreasing function and satisfies $(\bf{H})$ with $1<r<p_{s,\mu}^*$. 
	One can check that a solution $\underline{w}$  to 
	\begin{equation*}
	(-\Delta)_p^s u
	=\la |u|^{q-2}u , \; u>0  \; \text{in}\;
	\Om,
	u=0 \; \text{ in } \RR^N \setminus \Om,
	\end{equation*}
	is  subsolution to \eqref{ch33} for all $\la>0$ whereas for $\la>0$ small enough, a solution $\overline{w}$ to 
	\begin{equation*}
	(-\Delta)_p^s u
	=1 , \; u>0  \; \text{in}\;
	\Om,
	u=0 \; \text{ in } \RR^N \setminus \Om.
	\end{equation*}
	is a supersolution to \eqref{ch33} such that $\underline w\leq \overline w.$
	Now using Theorem \ref{thmch4.1}, we infer that there exists a solution, say $w_0$, to \eqref{ch33}, which is a local minimizer in $W_0^{s,p}(\Om)$.
	Then the existence of the second solution is guaranteed
	by showing that the energy functional associated to \eqref{ch33} satisfies the moutain pass geometry and Palaise-smale condition. 
$\hfill \square$                                                                           


\end{document}